\documentclass[12pt]{amsart}
\pdfoutput=1

\pdfpageattr{/Group <</S /Transparency /I true /CS /DeviceRGB>>}

\usepackage[T1]{fontenc}
\usepackage{hyperref}
\usepackage[autostyle=true]{csquotes}
\usepackage{amsfonts,amsmath, amsthm, amssymb, color}
\usepackage{graphicx}
\usepackage{caption, subcaption}
\usepackage{booktabs}

\usepackage{xspace}
\usepackage{pgf, pgfplots,tikz}
\usetikzlibrary{calc,shapes,decorations.markings,decorations.text}

\usetikzlibrary{arrows}
\usepackage{tikz-3dplot}
\usepackage{tabularx}
\usepackage{makecell}

\usepackage[width=6.6in, height=8.7in]{geometry}
\usepackage{array}
\newcolumntype{?}{!{\vrule width 1pt}}

\usepackage{algorithmicx}
\usepackage{algorithm}
\usepackage{algpseudocode}

\MakeRobust{\Call}

\algdef{SE}[DOWHILE]{Do}{DoWhile}{\algorithmicdo}[1]{\algorithmicwhile\ #1}
\algrenewcommand{\algorithmiccomment}[1]{\hfill $\rhd$ \emph{#1}}%{\hskip1em $\rhd$ \emph{#1}}
\algrenewcommand{\algorithmicrequire}{\textbf{Input:}}
\algrenewcommand{\algorithmicensure}{\textbf{Output:}}
\algnewcommand{\Or}{\textbf{or}}
\algnewcommand{\And}{\textbf{and}}
\algnewcommand{\Not}{\textbf{not}\,}
\algnewcommand\algorithmicforeach{\textbf{for each}}
\algdef{S}[FOR]{ForEach}[1]{\algorithmicforeach\ #1\ \algorithmicdo}

\usepackage[colorinlistoftodos,bordercolor=orange,backgroundcolor=orange!20,linecolor=orange,textsize=scriptsize]{todonotes}
\colorlet{critical}{red!20}
\colorlet{postponed}{blue!20}
\colorlet{minor}{green!20}

\newcommand\numberOfOrbits{14\,373\,645}
\newcommand\numberOfTriangulations{344\, 843 \,867}
\newcommand\idHampeJoswigTriangulation{5054117}
\newcommand\idHoneycomb{12369387}

\newcommand{\PP}{\mathbb{P}}
\newcommand{\RR}{\mathbb{R}}
\newcommand{\QQ}{\mathbb{Q}}

\newcommand{\ZZ}{\mathbb{Z}}
\newcommand{\NN}{\mathbb{N}}
\newcommand{\TP}{\mathbb{TP}}
\newcommand{\cP}{\mathcal{P}}

\newcommand{\motif}{\mathcal{R}}
\DeclareMathOperator{\seccone}{sec}

\newcommand\smallSetOf[2]{\{#1 \colon #2\}}

\newcommand{\vones}{{\bf 1}}
\newcommand\transpose[1]{{#1}^{\top}}
\DeclareMathOperator{\Trop}{Trop}
\DeclareMathOperator{\Pfaff}{Pfaff}
\DeclareMathOperator{\Gr}{G}
\DeclareMathOperator{\initial}{in}

\newtheorem{theorem}{Theorem}
\numberwithin{theorem}{section}
\newtheorem{proposition}[theorem]{Proposition}
\newtheorem{lemma}[theorem]{Lemma}

\theoremstyle{definition}

\newtheorem{example}[theorem]{Example}
\newtheorem{remark}[theorem]{Remark}

\newcommand\polymake{{\tt polymake}\xspace}
\newcommand\polydb{{\tt polyDB}\xspace}

\newcommand\scip{{\tt SCIP}\xspace}
\newcommand\nauty{{\tt nauty}\xspace}
\newcommand\atint{{\tt a-tint}\xspace}
\newcommand\macaulay{{\tt Macaulay2}\xspace}

\date{}

\author[M.\,Joswig]{Michael Joswig}
\address[M.\,Joswig]{TU Berlin\\Germany}
\email{joswig@math.tu-berlin.de}

\author[M.\,Panizzut]{Marta Panizzut}
\address[M.\,Panizzut]{TU Berlin\\Germany}
\email{panizzut@math.tu-berlin.de}

\author[B.\,Sturmfels]{Bernd Sturmfels}
\address[B.\,Sturmfels]{MPI Leipzig\\Germany and UC Berkeley\\USA}
\email{bernd@mis.mpg.de}% Sturmfels
 
\title[The Schl\"afli Fan]{The Schl\"afli Fan}

\begin{document}
\maketitle

\begin{abstract}
 Smooth tropical cubic surfaces are parametrized by maximal
cones in the unimodular secondary fan of the triple tetrahedron.
There are $\numberOfTriangulations$ such cones, organized into a database of
$\numberOfOrbits$ symmetry classes. The Schl\"afli fan gives a further
refinement of these cones. It reveals all possible patterns of lines on tropical cubic surfaces, thus serving
as a combinatorial base space for the universal Fano variety.
 This article develops the relevant theory and offers
a blueprint for the analysis of big data in tropical  geometry.
\end{abstract}

\maketitle

\section{Introduction}
A cubic surface in  projective $3$-space $\PP^3$ is the zero set of a cubic polynomial
%local_var_names<Polynomial<TropicalNumber<Min, Rational>, Int>>(qw(W X Y Z));
%for (my $i=0; $i<20; ++$i) { my $t=new Polynomial<TropicalNumber<Min, Rational>,Int>(new Vector<TropicalNumber<Min>>([0]),new Matrix<Int>([$S->MONOMIALS->[$i]])); print " + C_$i*$t" }
\begin{equation} \label{eq:cubicf}
  \begin{split}
  c_0&w^3 + c_1w^2z + c_2wz^2 + c_3z^3 + c_4w^2y + c_5wyz + c_6yz^2
+ c_7wy^2 + c_8y^2z + c_9y^3  +  c_{10}w^2x \\  &+ \,  
c_{11}wxz + c_{12}xz^2 + c_{13}wxy + c_{14}xyz + c_{15}xy^2 + c_{16}wx^2 + c_{17}x^2z + c_{18}x^2y +
c_{19}x^3.
  \end{split}
\end{equation}
Here $(w:x:y:z)$ are homogeneous coordinates on $\PP^3$.
 George Salmon and Arthur~Cayley discovered in the
1840s that every smooth cubic surface contains $27$ lines.
  Ludwig Schl\"afli studied the combinatorics of the lines in his 1858 article  \cite{Schl}.
The name of that Swiss mathematician appears in our title.

This article is dedicated to the memory of Branko Gr\"unbaum.
Gr\"unbaum is famous for his work on polytopes and arrangements,
especially those that admit a high degree of symmetry. In the literature
on these geometric figures, one sees a direct line connecting
Ludwig Schl\"afli to Branko Gr\"unbaum. This is highlighted
by the use of the \emph{Schl\"afli symbol} for  symmetries
of polyhedra. 

The combinatorial strand of algebraic geometry underwent
a major shift during the past two decades, thanks to the
advent of tropical geometry  \cite{TropicalBook}.
The following question emerged early on during the tropical revolution:
\emph{What are all shapes of smooth cubic surfaces in tropical $3$-space,
and which arrangements of tropical lines occur on such surfaces?}
A first guess is that there are $27$ lines, just like in the classical case.
But this is false. Vigeland~\cite{VigelandInfinite}
showed that the number of lines can be infinite.
A textbook reference is  \cite[Theorem~4.5.8]{TropicalBook}.

The aim of this article is to give a comprehensive answer to the questions above.
We will do so via a computational study of  all smooth tropical cubic surfaces.
These surfaces are dual to  unimodular regular triangulations of the 
\emph{triple tetrahedron} $3 \Delta_3$, which is the Newton polytope of 
the cubic polynomial seen in (\ref{eq:cubicf}).
The relevant definitions will be reviewed in Section \ref{sec:subdivisions}.

Our point of departure is the article \cite{PV}, which classifies 
the ten motifs that describe the potential positions of a tropical line
on a cubic surface. These motifs are  denoted 3A, 3B, $\ldots\,$, 3J.
They are shown in Table \ref{Table:dualmotif}.
The advance we report in this paper is a large-scale computation
that identifies the motifs of all lines that actually occur on the many tropical smooth cubic surfaces.

Our contribution rests on earlier work by Jordan et al.~\cite{mptopcom:paper}
who developed highly efficient tools for enumerating triangulations.
Their count for $3 \Delta_3$ in
\cite[Theorem 19]{mptopcom:paper} shows
that  there are $\numberOfOrbits$ combinatorial types of 
smooth tropical cubic surfaces. Here, the types are the orbits of  the 
symmetric group $S_4$ by permuting $w,x,y,z$ in the $20$ terms of~(\ref{eq:cubicf}).
 Adding up the sizes of all $S_4$-orbits, we obtain the total number 
 $\numberOfTriangulations$ of smooth tropical cubics.
 
This article is organized as follows.
In Section~\ref{sec:subdivisions} we fix notation, we discuss unimodular triangulations
of the tetrahedron $3 \Delta_3$, and we review basics on lines and surfaces in tropical projective space $\mathbb{TP}^3$.
We also recall the classification of motifs in \cite{PV}.
Section~\ref{sec:software} furnishes our classification of smooth tropical cubic surfaces.
This is presented in Theorem~\ref{thm:censustriang}, and it is followed by a detailed
explanation of the methodology that underlies our work and its results.
 
Section \ref{sec:motifs} studies occurrences of motifs 
in the unimodular triangulations of $3 \Delta_3$.
Our main result is Theorem \ref{thm:114}. We present an
algorithm for computing occurrences. This rests on several lemmas
that describe geometric constraints.  The algorithm is applied to
all triangulations in Theorem~\ref{thm:censustriang}.
As a consequence, we get a complete list of  occurrences of motifs for each of the $\numberOfOrbits$ types.

In Section \ref{sec:schlaefli} we zoom in on particular secondary cones.
For each cubic surface of one type, an occurrence of a motif may be visible or not.
Being visible means that there exists a line for that motif.
Hence, for any specific surface,
only a subset of the motifs occurring in the triangulation is visible.
The regions on which that subset is constant are convex polyhedral cones.
These form the Schl\"afli fan.
Thus, each of the $\numberOfOrbits$  secondary cones is divided
into its Schl\"afli cones. We present and discuss 
the result of that computation.

Our combinatorial and computational study in this paper
lays the foundation for future work on the nonarchimedean geometry
of classical cubic surfaces over a valued field.
In Section \ref{sec:fano} we take a step into that direction.
We discuss the universal Fano variety and the universal Brill variety,
and we examine the tropical discriminants of these universal families.
The first version of this article had a Section 7 which
proposed a normal form for cubic surfaces, called the eight-point model.
This was deleted in this final version because an even better such model
was found in the subsequent project \cite{sicily} with Emre Sert\"oz.

The methods from computer algebra and polyhedral geometry which led to our results are at the forefront of what is currently possible in terms of hardware, algorithms and software.
For instance, to determine and analyze the regular unimodular triangulations of $3\Delta_3$ took more than 200 CPU days on an Intel Xeon E5-2630 v2 cluster.
Yet the most difficult question we had to answer was how to make the results of such a large computation available to 
others.
For this we set up a \polymake extension \texttt{TropicalCubics} \cite{extension:tropcubics} and a database within the \polydb framework \cite{polydb:paper}.
They can be accessed via \polymake \cite{DMV:polymake}. The database can also be used via  an independent API.
We believe that this approach can serve as a model for sharing \enquote{big data} in mathematical research.

\section{Triangulations, Cubic Surfaces and Tropical Lines}
\label{sec:subdivisions}

In this section we review the basics and known results on which our study rests.
For conventions on tropical geometry we follow the textbook by
Maclagan and Sturmfels \cite{TropicalBook}. Our tropical semiring is
the min-plus algebra $(\RR \cup \{\infty\}, \oplus, \odot)$.
We use upper case letters to denote tropical variables and coefficients. 
Our orderings of variables and monomials are consistent with the 
conventions used by \polymake \cite{DMV:polymake}.
For instance, here is a homogeneous tropical cubic polynomial:
% joswig@priort: data (master)> bunzip2 -c 3d3_reg_unimodular.cf-stat.bz2 | sed -n 5054117p
% 5054117 : 44 0 1 15 19 0 9 2 4 0 38 0 15 16 4 1 33 16 14 29 : 260 44
\begin{equation}\label{eq:typical:trop-polynomial}
  \begin{split}
  44&W^3 \oplus W^2Z \oplus 1WZ^2 \oplus 15Z^3 \oplus 19W^2Y \oplus WYZ \oplus 9YZ^2 \\
  & \oplus 2WY^2 \oplus 4Y^2Z \oplus Y^3 \oplus 38W^2X \oplus WXZ \oplus 15XZ^2 \oplus 16WXY \\
  & \oplus 4XYZ \oplus 1XY^2 \oplus 33WX^2 \oplus 16X^2Z \oplus 14X^2Y \oplus 29X^3.
  \end{split}
\end{equation}
 The expression (\ref{eq:typical:trop-polynomial})
 is evaluated in classical arithmetic as follows:
\[ \min \bigl\{44 {+} 3 W, \,2W {+} Z,\, 1 {+}W {+} 2Z, \, 
15 {+} 3Z, \, \ldots\,,\,14 {+} 2X {+} Y, \, 29 {+} 3 X \bigr\}. \]
The surface defined by (\ref{eq:typical:trop-polynomial}) is the
set of all points $(W,X,Y,Z)$ for which this minimum is attained at least twice.
That tropical cubic surface lives in the tropical projective torus $\RR^4/\RR\vones$,
but it also has a natural compactification in the
tropical projective space $\TP^3$. The latter is described in \cite[Chapter 6]{TropicalBook}.

A standard reference for the material that follows next is the textbook by De Loera, Rambau and Santos \cite{book:triangulations}. 
Reading the coefficients of the tropical polynomial as a height function defines a regular 
polyhedral subdivision of the $20$ lattice points in $3 \Delta_3$.
If the coefficients are generic enough then the dual subdivision is a triangulation.
For now the latter property may be taken as a definition for \emph{generic}; it is a main point of 
later sections to refine this.
If each of its tetrahedra has unit normalized volume, then the triangulation is \emph{unimodular} and the tropical cubic surface is  \emph{smooth}.
Every unimodular triangulation $T$ of the configuration
$3\Delta_3$ has the same  f-vector $f(T) = (20,64,72,27)$.
Its boundary has the f-vector $f(\partial T) = (20,54,36)$. From this we conclude that
every smooth tropical cubic surface has $27$ vertices, 
$36$ edges, $36$ rays, $10$ bounded 2-cells, and $54$ unbounded 2-cells.
This is the case $d=3$ in \cite[Theorem 4.5.2]{TropicalBook}.
Specifically, the $64-54=10$ interior edges of $T$ correspond to the bounded polygons in the surface.
These $10$ polygons form the bounded complex of the tropical surface.
This is also known as the \emph{tight span}.
For cubics, it is contractible. We define the \emph{B-vector} of
the triangulation $T$ to be $(b_3,b_4,b_5,\ldots)$,
where $b_j$ denotes the number of $j$-gons in the tight span. 
The \emph{GKZ-vector} is $(g_0,g_1,\ldots,g_{19})$, where $g_i$ is the number of tetrahedra containing point~$i$.

\begin{example} \label{ex:introexample}
The tropical cubic polynomial in \eqref{eq:typical:trop-polynomial} is identified with its coefficient vector
$(44, 0, 1, 15, 19, 0, 9, 2, 4, 0, 38, 0, 15, 16, 4, 1, 33, 16, 14, 29)$. 
This defines a unimodular triangulation $T$ of $3 \Delta_3$.
Its $27$ tetrahedra are given by their labels:
% _id=5054117
\begin{equation}\footnotesize
\label{eq:typical:triangulation}
\begin{array}{l}
\{0,1,4,10\}, \, 
\{1,2,5,11\}, \, 
\{1,4,7,13\}, \, 
\{1,4,10,16\}, \, 
\{1,4,13,19\}, \, 
\{1,4,16,19\}, \, 
\{1,5,9,11\}, \, \\
\{1{,}7{,}9{,}15\}, \, 
\{1{,}7{,}13{,}18\}, \, 
\{1{,}7{,}15{,}18\}, \, 
\{1{,}9{,}11{,}15\}, \, 
\{1{,}11{,}15{,}18\}, \, 
\{1,11,18,19\}, \, 
\{1,13,18,19\}, \, \\ 
\{2,3,6,14\}, \, 
\{2,3,11,14\}, \,  
\{2,5,9,11\}, \, 
\{2,6,8,14\}, \, 
\{2,8,9,14\}, \, 
\{2,9,11,15\}, \, 
\{2,9,14,15\}, \, \\
\{2,11,14,15\}, \, 
\{3,11,12,14\}, \, 
\{11,12,14,17\}, \, 
\{11,14,15,17\}, \, 
\{11,15,17,18\}, \, 
\{11,17,18,19 \}.
\end{array}
\end{equation}

The GKZ-vector equals
$(1, 14, 9, 3, 5, 3, 2, 4, 2, 7, 2, 14, 2, 4, 9, 9, 2, 4, 7, 5)$.
The last entry $5$ means that the label
$19$ occurs five times in  (\ref{eq:typical:triangulation}).
The B-vector of \eqref{eq:typical:triangulation} is  $(2, 4, 2, 2)$.
To see this, we list the ten interior edges and their links:
\[ \footnotesize
\begin{matrix}
\! \{1, 13\} , [4, 7, \! 18, \! 19] & 
\{1, 15\}, [7, 9, 18, 11]  & \!
\{1, 18\}, [7, 13, 19, 11, 15] &
\{2, 14\}, [3, 6, 8, 9, 15,11] & 
\! \{2, 15\}, [9, \! 11, \! 14]  \\
\! \{9, 11\}, [1, 5, 2, 15] &
\! \{11, 18\}, [1, \! 15, \! 17, \! 19] & 
\{11, 14\}, [2, 3, 12, 17, 15] & 
\! \! \{11, \! 15\}, [1, 9, 2, \! 14, \! 17,\! 18] &
\{5, 11\}, [1, 2, 9]  
\end{matrix}
\]

The \emph{link} of an edge $e$ in $T$ is the graph of all edges in $T$ whose union with $e$ is a tetrahedron in $T$.
If $e$ is an interior edge of $T$, then this graph is a cycle.
For instance, the link of $\{9,11\}$ 
is the $4$-cycle $\{\{1,5\}$, $\{5,2\}$, $\{2,15\}$, $\{15,1\}\}$.
The corresponding bounded $2$-cell in the tropical cubic surface is a quadrilateral.
The triangulation \eqref{eq:typical:triangulation} lies in the same $S_4$-orbit as the one featured in \cite[\S6.2]{HampeJoswig:2017}.
\end{example}

Each of the $36$ bounded edges of the surface determines a linear inequality
among the coefficients $C_0,C_1,\ldots,C_{19}$, expressing that the edge has positive length.
The \emph{secondary cone} $\seccone(T)$ is the set of solutions to these inequalities.
This is a full-dimensional cone in $\RR^{20}$ with $4$-dimensional lineality space.
The number of facets of $\seccone(T)$ is between $16$ and $36$.
The secondary cone of the triangulation \eqref{eq:typical:triangulation} has $16$ facets.
It contains the coefficient vector of~\eqref{eq:typical:trop-polynomial}.

The symmetric group $S_4$ acts naturally on the $20$ points in $3\Delta_3$.
This induces an action on the set of all triangulations.
Note that $S_4$ also acts on the set of GKZ-vectors.
The $S_4$-orbit of the triangulation $T$ from \eqref{eq:typical:triangulation} has size 24.
Equivalently, the stabilizer of $T$ is trivial.
The census of unimodular triangulations and associated cubic surfaces
is presented in Theorem~\ref{thm:censustriang}.

\begin{table}
\caption{The ten motifs from \cite{PV} for tropical lines on generic  cubic surfaces.}\label{Table:dualmotif}
\begin{center}
\begin{small}
%\resizebox{\textwidth}{!}{
\begin{tabular}[t]{ccc}
\toprule
Marked Lines & Associated Motifs & Necessary Conditions \\
\midrule
\multicolumn{3}{c}{Isolated Lines}\\
\midrule
\includegraphics{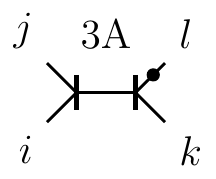} &\includegraphics{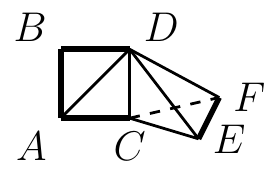} &  \raisebox{+\height}{\thead{Exits: $AB\subseteq F_i, \ \ \ BD\subseteq F_j,$ \\ $\qquad \ \ \ AC \subseteq F_k, \ \ \ EF \subseteq F_l,$ \\ $AD \subseteq \{x_i+x_j=1\},$ $CD \subseteq \{x_l=1\},$ \\  $A\not = E, \,F$ and  $B\not = C$. }}  \\ 

\includegraphics{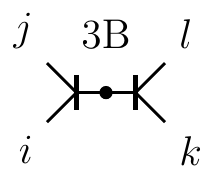} &\includegraphics{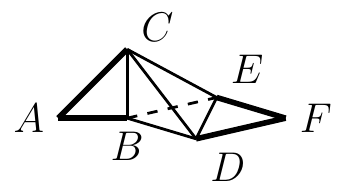} &\raisebox{+\height}{\thead{Exits: $AB\subseteq F_i, \ \ \ AC\subseteq F_j,$ \\ $ \qquad \ \ \ DF \subseteq F_k, \ \ \ EF \subseteq F_l,$\\ $BC \subseteq \{x_i+x_j=1\}$, $ DE \subseteq \{x_k + x_l=1\}$, \\ $A \not = D,E$ \ \  $F\not = B,C$ and $A\not=F$.  }}\\

\includegraphics{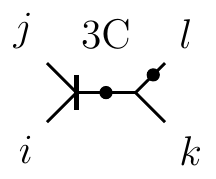}   &\includegraphics{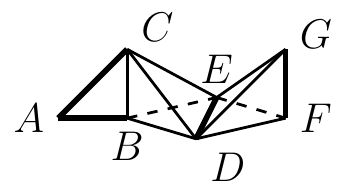} & \raisebox{+\height}{\thead{Exits: $AB\subseteq F_i, \ \ \ AC\subseteq F_j,$ \\ $ \qquad \ \ \ DE \subseteq F_k, \ \ \ FG \subseteq F_l,$ \\ $BC \subseteq \{x_i+x_j=1\},$ $ DE \subseteq \{x_l=1\}\cap F_k$, \\ $A\not =D,E$. }}\\

\includegraphics{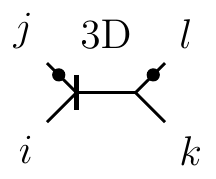}&\includegraphics{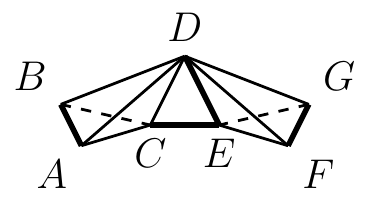}&\raisebox{+\height}{\thead{ Exits: $CE\subseteq F_i, \ \ \ AB\subseteq F_j,$ \\ $ \qquad \ \ \ DE \subseteq F_k, \ \ \ FG \subseteq F_l,$\\ $CD \subseteq \{x_j=1\},$ $DE \subseteq \{x_l=1\}\cap F_k$, \\ $E \not = A,B$.}} \\

\includegraphics{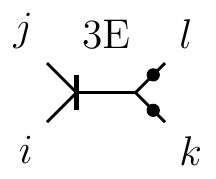} &\includegraphics{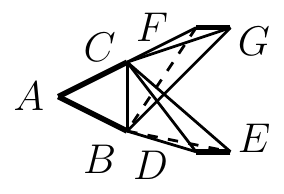}&\raisebox{+\height}{\thead{Exits: $AB\subseteq F_i, \ \ \ AC\subseteq F_j,$ \\ $ \qquad \ \ \ DE \subseteq F_k, \ \ \ FG \subseteq F_l,$ \\$BC \subseteq \{x_k=1\} \cap \{x_l=1\}$.  }} \\

\includegraphics{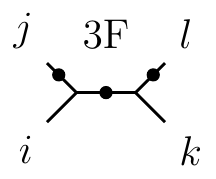} &\includegraphics{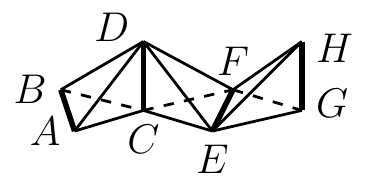}& \raisebox{+\height}{\thead{Exits: $CD\subseteq F_i, \ \ \ AB\subseteq F_j,$ \\ $ \qquad \ \ \ EF \subseteq F_k, \ \ \ GH \subseteq F_l,$\\ $CD \subseteq \{x_j=1\} \cap F_i,$  \ \ $ EF \subseteq \{x_l=1\} \cap F_k$. }}\\

\includegraphics{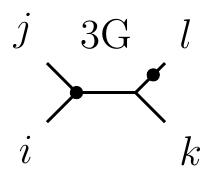} &\includegraphics{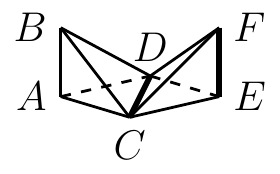}& \raisebox{+\height}{ \thead{Exits: $CD\subseteq F_k, \ \ \ EF\subseteq F_l,$ \\ $ABCD$ has exits also in $F_i$ and $F_j$, \\$CD \subseteq \{x_l=1\} \cap F_k$.} }\\

\includegraphics{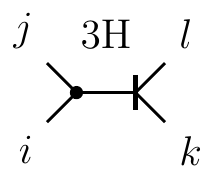} &\includegraphics{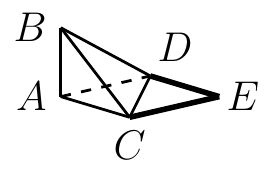}&\raisebox{+\height}{\thead{Exits: $CE\subseteq F_k, \ \ \ DE\subseteq F_l,$ \\ $ABCD$ has exits also in $F_i$ and $F_j$, \\$CD \subseteq \{x_k+x_l=1\}$,  \ \ $E \not =A,B$. }} \\
\midrule
\multicolumn{3}{c}{Families of Lines}\\
\midrule
 \includegraphics{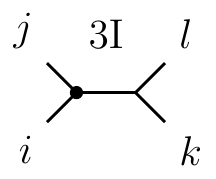}&\includegraphics{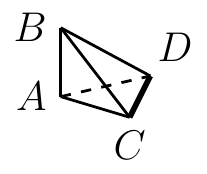} &\raisebox{+\height}{\thead{Exits: $CD\subseteq F_k \cap F_l,$ \\ $ABCD$ has  exits also in $F_i$ and $F_j$.}} \\

 \includegraphics{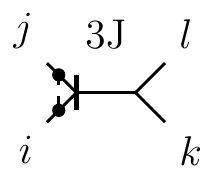}&\includegraphics{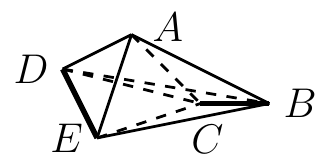} &\raisebox{+\height}{\thead{Exits: $BC\subseteq F_i \cap F_j$ \\
$\qquad \ \ \ DE \subseteq F_k \cap F_l$, \\ $AD \subseteq \{x_j=1\}, AE \subseteq \{x_i=1\}.$}}\\
\bottomrule
 \end{tabular}
 %}
 \end{small}
\end{center}
\end{table}

We now come to tropical lines in 3-space.
Vigeland started the classification of how such lines can lie on generic smooth tropical cubic surfaces.
Based on a massive random search with \polymake, Simon Hampe realized that the classification was not complete.
The triangulation \eqref{eq:typical:triangulation} occurred in the joint article \cite{HampeJoswig:2017} as the first explicit counter-example to Vigeland's list.
The final characterization is due to Panizzut and Vigeland \cite{PV}.
Their list of ten motifs is reproduced in Table \ref{Table:dualmotif}.
This table forms the foundation for our present study.

We identify $\RR^3$ with  $\RR^4 / \RR {\bf 1}$ by setting
$\omega_0 = -(e_1+e_2+e_3)$,  $\,\omega_1 = e_1$, \,$\omega_2 = e_2$ and $\omega_3 = e_3$.
A \emph{tropical line in $\RR^3$} is a balanced polyhedral complex given by two $3$-valent adjacent vertices, joined by one bounded edge, and four rays with directions $\omega_0$, $\omega_1$, $\omega_2$ and $\omega_3$.
If the bounded edge has length zero, the tropical line is  \emph{degenerate}.
Non-degenerate lines come in three labeled types, given by the
direction of the bounded edge. This direction is either
$\omega_0 + \omega_1 = - \omega_2 - \omega_3$ \ or \
$\omega_0 + \omega_2 = - \omega_1 - \omega_3$ \ or \
$\omega_0 + \omega_3 = - \omega_1 - \omega_2$. We denote these three
types  by $01|23$, $02|13$ and $03|12$.
This is shown in Figure~\ref{fig:line-types}.

\begin{figure}[tbh]
\centering
\includegraphics{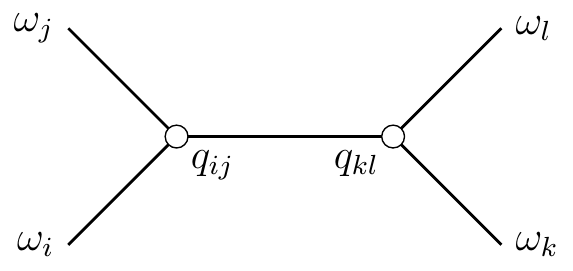}
\caption{A non-degenerate tropical line of labeled type $ij|kl$ in $3$-space.}
\label{fig:line-types}
\end{figure}

Each tropical line $L$ in $\RR^3$ is encoded (up to tropical scaling) by its \emph{tropical Pl\"ucker vector} 
$\, P=(P_{01},P_{02},P_{03},P_{12},P_{13},P_{23})   \in \RR^6$.
The six $P_{ij}$ are the tropical $2 \times 2$ minors of a $2 \times 4$-matrix.
A vector $P \in \RR^6$ is the tropical Pl\"ucker vector of a line 
if and only if it lies on the tropical hypersurface given~by
\begin{equation}\label{eq:trop-line:pluecker}
  P_{01}\odot P_{23} \ \oplus\ P_{02} \odot P_{13} \ \oplus\ P_{03} \odot P_{12}.
\end{equation}
This means that the minimum in \eqref{eq:trop-line:pluecker} is attained at least twice.
Equivalently, $P$ is a height function on the six vertices of the regular octahedron
   which induces a split into two Egyptian pyramids
\cite[Figure 4.4.1]{TropicalBook}.
The tropical hypersurface defined by (\ref{eq:trop-line:pluecker}) is the
tropical Grassmannian Trop(G$^0(2,4)$).

A tropical line $L$ is recovered from its Pl\"ucker vector $P \in \RR^6$ as follows.
We start by identifying the pair of terms in \eqref{eq:trop-line:pluecker} 
which attains the minimum. Suppose  $P_{01}+P_{23} = P_{02}+P_{13} \leq P_{03}+P_{12}$,
i.e., the labeled type is $03|12$.
Then, by \cite[Example~4.3.19]{TropicalBook}, $L$ consists of the segment joining the two points
\begin{equation}\label{eq:trop-line:vertices}
  \begin{aligned}
    q_{03} &\,=\, (P_{02}+P_{03}, P_{02}+P_{13}, P_{02}+P_{23}, P_{03}+P_{23}) \quad \text{and}\\
    q_{12} &\,=\, (P_{02}+P_{13}, P_{12}+P_{13}, P_{12}+P_{23}, P_{13}+P_{23}) \quad \, \text{in}\,\,
     \,\,  \RR^4 / \RR {\bf 1}
  \end{aligned}
\end{equation}
and the four rays $q_{03}+\RR_{\geq 0}\cdot\omega_0$, $q_{03}+\RR_{\geq0}\cdot\omega_3$, $q_{12}+\RR_{\geq 0}\cdot\omega_1$, $q_{12}+\RR_{\geq0}\cdot\omega_2$.
The formulas for the other two labeled types, $01|23$ and $02|13$, are analogous.

In summary, the vertices $q_{ij}$ and $q_{kl}$ of a tropical line $L$ are computed from the
Pl\"ucker coordinates in (\ref{eq:trop-line:vertices}). Conversely, the Pl\"ucker vector
is obtained by taking the tropical $2 \times 2$ minors of the $2 \times 4$-matrix
with rows $q_{ij}$ and $q_{kl}$. 

The article \cite{PV} describes the various ways in which a tropical line $L$
can lie on a smooth cubic surface $S$ in $3$-space.
Here we require $S$ to be generic in the precise sense of Section 5.  
On the line $L$ we mark the points where $L$ intersects edges or vertices of
the surface $S$.
These are the bars and dots indicated on the tropical lines in the left column of Table~\ref{Table:dualmotif}.
Each bar is dual to a triangle in $T$, and each dot is dual to a tetrahedron in~$T$.
Formally, a \emph{motif} of a tropical cubic surface is one of the ten abstract simplicial complexes 
3A, 3B, $\ldots \,$, 3J which are listed in the middle column of Table \ref{Table:dualmotif}.
Each is equipped with a labeling of its vertices by $A,B,\ldots$ and a marking of precisely four edges by $i,j,k,l$. 
That this list of ten motifs is complete is the main result of~\cite{PV}.

The number of vertices of the ten motifs range between four and eight; the marked edges are the \emph{exits} of the motif.
The names of the motifs
all start with the digit 3 to indicate the degree of the tropical surface; there are more motifs for other degrees \cite[Table~2]{PV}.
The article \cite{PV} distinguishes between \enquote{primal motifs} and \enquote{dual motifs}.
We use the term motif for what is called \enquote{dual motif} in \cite{PV}.  
Our Table~\ref{Table:dualmotif} uses
$x_i, x_j, x_k, x_l$ for the homogeneous coordinates of the lattice points in $3\Delta_3$,
and it uses the notation $\,F_i = \{ x \in 3 \Delta_3 \,|\, x_i = 0 \}\,$ for the facets
of $3\Delta_3$. The third column of Table \ref{Table:dualmotif} explicates
additional conditions to be satisfied by some edges
in order for the motif to occur in $T$. These~are derived in 
\cite[Proposition 23]{PV}. 
They will become important in Section~\ref{sec:motifs}.

\section{Data, Software, and Lines on Cubics}
\label{sec:software}

A primary goal of the present work is to present a database for smooth tropical cubic surfaces.
We now explain our database and the underlying methodology.
We start with the classification of combinatorial types.
The proof of this result is the computation reported in \cite[Theorem 19]{mptopcom:paper}, plus an analysis of the orbits.

\begin{theorem} \label{thm:censustriang}
  The triple tetrahedron $3 \Delta_3$ has precisely $\numberOfTriangulations$ regular unimodular triangulations.
  These are grouped into $\numberOfOrbits$ orbits with respect to the natural action of $S_4$.
  The distribution of orbit sizes is shown in Table~\ref{tab:orbit-sizes}.
\end{theorem}

\begin{table}[htb]
  \centering
  \caption{Distribution of orbit sizes among smooth tropical cubic surfaces:
  $99.93\%$ of the combinatorial types have no symmetry, i.e., the orbit size is~$24$.
  }
  \label{tab:orbit-sizes}
  \begin{tabular*}{.75\linewidth}{@{\extracolsep{\fill}}cccccc@{}}
    \toprule
    3 & 4 & 6 & 8 & 12 & 24 \\
    \midrule
    3 & 15 & 25 & 82 & 10124 & 14363396 \\
    \bottomrule
    \smallskip
  \end{tabular*}
\end{table}

\begin{remark}
  Each smooth tropical cubic surface in $\RR^4/\RR\vones$ 
  has four elliptic  curves in its boundary  in $\TP^3$.
  These are the tropical plane cubics which are dual to the induced triangulations 
  of the ten lattice points in the triple triangle $3\Delta_2$. 
  That configuration has   precisely $79$ unimodular triangulations, all of which are regular.
  They are grouped into $18$ orbits with respect to the natural action of $S_3$. 
  Hence, we encounter at most $79^4=38\,950\,081$ triangulations of the boundary $\partial(3\Delta_3)$.
  This means that, on the average, more than eight regular unimodular triangulations of $3\Delta_3$ induce the same boundary triangulation.
\end{remark}

Before we enter the technical details, we briefly pause to reflect on the nature of a result like Theorem~\ref{thm:censustriang}, how it can be useful, and to what extent it can be trusted.
Theorem~\ref{thm:censustriang} is a highly condensed statement which was derived from massive computations, partially on large clusters, and the total time spent exceeds several months.
Most readers will not have access to these types of hardware and technical
resources and therefore will be unable to repeat these computations on their own.
As we see it, the bulk of the data is the actual theorem.
Theorem~\ref{thm:censustriang} is a mere corollary which follows from something which is too large to write down in any article.
That data and more is made publically available at
\begin{equation} \label{urldata}
  \hbox{\url{https://db.polymake.org/}}
\end{equation}
to allow everyone to derive their own corollaries.
We stress that \emph{all} the software that was used in the process is open source.
Therefore, the entire proof of Theorem~\ref{thm:censustriang}, which consists of software and data (in addition to this text), is available for scrutiny.
Ideally, such a computer proof would be formalized, but currently this seems to be out of scope for a project of this size.
Turning this into a formal proof would be a large project on its own, probably much larger than \texttt{flyspeck} \cite{flyspeck}, if feasible at all.
This leaves the question of correctness.

As we see it, making data available and documenting this in an article is a 
necessary first step.
Everyone is invited to probe the data for its correctness; we prepared various tools, explained below, to help with the probing.
Any errors found in the future will be corrected in the data base.
It would be desirable to have a general mechanism for this, accepted by the mathematical community.
Finally, we would like to point out that it was a massive \polymake experiment run by Simon Hampe which lead to the triangulation \eqref{eq:typical:triangulation}, which exhibited a flaw in a first version of \cite{PV}.
That may be seen as a predecessor to this project.

\bigskip

\emph{High-level view on the data computed.}  For each of the $\numberOfOrbits$ triangulations $T$ in our database, the following annotations are reported: 
the GKZ-vector, the B-vector, the orbit size with respect to the $S_4$-action, and a unique \emph{identifier}.
The identifier is an integer between $1$ and $14\,373\,645$, which can be used to retrieve the triangulation 
and data derived. Frequently we will use the symbol `\#' for marking identifiers.
The triangulation \eqref{eq:typical:triangulation} has the identifier $\#\idHampeJoswigTriangulation$.

The facets of each triangulation are listed in lexicographic order. The representative for a
combinatorial type is chosen such that the GKZ-vector is lexicographically minimal.
Another important item in our database is a vector $C \in \NN^{20}$ 
of minimum coordinate sum in the interior of each secondary cone.
In order to find this vector, we had to solve an integer linear programming problem.
We did this using the software \scip \cite{scip}.
The coefficients of the tropical polynomial \eqref{eq:typical:trop-polynomial} 
were derived from the triangulation \eqref{eq:typical:triangulation} in this way.
Note that, by construction, $C$ is always generic in the sense that the regular subdivision induced is a triangulation.
However, it is \emph{not} generic as defined in Section~\ref{sec:schlaefli}.

\bigskip

\emph{Exploring the database.} We now describe how to access the data we produced.
We offer a collection \texttt{SchlaefliFan} within the database \texttt{Tropical} of \polydb~\cite{polydb:paper}.
The simplest possible access is by directing a standard web browser to (\ref{urldata}).
 However, for best results, we recommend the concurrent use of a recent version of \polymake~\cite{DMV:polymake}.
The new \polymake extension \texttt{TropicalCubics} \cite{extension:tropcubics} is the software companion to 
this paper. It is available from and further explained at \url{https://polymake.org/doku.php/extensions/tropicalcubics}.
Future additions will deal with other aspects of tropical cubic surfaces.

\begin{table}[th]\centering
  \caption{Data for some triangulations.  The first row is
  the triangulation \eqref{eq:typical:triangulation}, and the second one is the honeycomb triangulation from Example~\ref{ex:honeycomb:motif3D} below.  The next two are combinatorially non-isomorphic but share the same canonical hash values.  The final two are combinatorially isomorphic but in distinct orbits.}
  \label{tab:data}
  \begin{tabular*}{.75\linewidth}{@{\extracolsep{\fill}}ccc@{}}
    \toprule
    Identifier & Canonical Hash & Altshuler Determinant\\
    \midrule
    $\#\idHampeJoswigTriangulation$ & $81\,541\,384$ & $614\,912$\\[1.5ex]
    $\#\idHoneycomb$ & $1\,464\,729\,205$ & $0$ \\[1.5ex]
    $\#1957163$ & $1\,000\,016\,429$ & $278\,528$\\
    $\#3315847$ & $1\,000\,016\,429$ & $684\,032$\\[1.5ex]
    $\#10720721$ & $1\,000\,063\,702$ & $560\,512$\\
    $\#14051499$ & $1\,000\,063\,702$ & $560\,512$\\
    \bottomrule
  \end{tabular*}
\end{table}

One pertinent question is how to find a given triangulation $T$ in the database.
The user is unlikely to know the search key, and $T$ may be
given by its list of facets as in \eqref{eq:typical:triangulation}.
One way is to compute the GKZ-vector and to then generate the lexicographically minimal representative within its $S_4$-orbit.
This is the preferred method since it identifies the regular triangulation uniquely.
Thus, in practice, the lex-minimal GKZ-vector works as another search key.
An alternative method is to find a canonical form of $T$ as a simplicial complex.
This means identifying the isomorphism type of the incidence graph of the $20$ vertices and the 
$27$ tetrahedra. The software  \nauty \cite{nauty} is a standard tool for this task.
It computes a \emph{canonical hash} value, which is a 64-Bit integer that encodes the 
isomorphism type.
This hash value is also stored in our database. It can be used as an index to retrieve a triangulation
instantly; cf.\ Table~\ref{tab:data}.
 
The canonical hash value is a combinatorial invariant, but it is not unique.
Table~\ref{tab:data} shows two triangulations with the same hash value.
Nonetheless, they are not isomorphic as abstract simplicial complexes, as can be seen as follows.
Let $v_1,v_2,\dots,v_k$ and $t_1,t_2,\dots, t_l$ be an ordering of the vertices and the facets, respectively, of a simplicial complex $T$.
The \emph{incidence matrix} $J$  is the $0/1$-matrix with $J_{ij}=1$ if vertex $v_i$ lies on the facet $t_j$ and $J_{ij}=0$ otherwise. We define
the \emph{Altshuler determinant} of~$T$ to be
$
  \max\bigl(\bigl|\det (J\transpose{J})\bigr|, \bigl|\det (\transpose{J} J)\bigr| \bigr). 
$
This number does not depend on the orderings \cite[Theorem~3]{AltshulerSteinberg:1973}.
It is a combinatorial invariant of $T$.
This distinguishes the third and fourth triangulations in Table~\ref{tab:data}.
Our database can be queried for Altshuler determinants directly.

It also happens that two abstractly isomorphic triangulations lie in different $S_4$-orbits.
A pair of examples is given at the end of Table~\ref{tab:data}.
Altogether there are $79\,572$ hash values (i.e., about $0.5\%$) that  correspond to
two or more $S_4$-orbits of triangulations.
The maximal multiplicity of any hash value is four.
So, with high probability, \nauty identifies the triangulation uniquely.

\bigskip

\emph{Lines in surfaces.} We now shift gears, with a discussion of the following basic problem.
Given a non-degenerate tropical line $L$ and a tropical cubic surface $S$, decide whether $S$ contains $L$.
We present an algorithm  that solves this.

Let $\ell(t) = [\ell_0(t),\ell_1(t),\ldots,\ell_m(t)]$ be an ordered list
of linear polynomials   $\ell_i(t) = \alpha_i t + \beta_i$.
An interval $U$ in $\RR$ is \emph{covered by $\ell(t)$} if the minimum value in the list
$\ell(u)$ is attained at least twice for all $u \in U$. This can only happen if
some $\ell_i(t)$ appear multiple times in $\ell(t)$.
We introduce  the \emph{coincidence partition}
\begin{equation}
\label{eq:partition1}
 \{0,1,\ldots,m\} \,=\,
\sigma_1 \, \dot \cup \, \sigma_2 \,  \dot\cup \, \cdots \, \dot\cup \, \sigma_r ,
\end{equation}
where ($i \in \sigma_k$ and  $j \in \sigma_l$) 
implies ($\ell_i = \ell_j$ if and only if $k=l$).
We write $\ell_{\sigma_k}(t)$ for the linear function
$\ell_i(t)$ with $i \in \sigma_k$.
The tropical polynomial function $\,\RR \rightarrow \RR^{m+1},
\, t \mapsto \min \, \ell(t)\,$
defines a partition into smaller intervals,
\begin{equation}
\label{eq:partition2}
 U \,= \, U_1 \cup U_2 \cup \cdots \cup U_s,
 \end{equation}
 with the following property:
on each $U_i$  precisely one function $\ell_{\sigma_{k(i)}}$
attains the minimum among our $r$ linear functions. Then $\ell(t)$ covers $U$ if and only~if
\begin{equation}
\label{eq:partition3}
|\sigma_{k(i)}| \geq 2
\quad \hbox{for all} \quad
i \in \{1,2,\ldots,s\}. 
\end{equation}

Our discussion translates into an algorithm called
the \emph{Covering Subroutine}. Its input is an interval $U$ in $\RR$
and a list $\ell(t)$ of linear polynomials, and its output is a yes-no decision whether
$U$ is covered by $\ell(t)$. In the no-case, the Covering Subroutine
also outputs a rational number $u \in U$ such that the minimum
in $\ell(u)$ is attained only once. In the yes-case, the Covering Subroutine outputs
the list of index sets $\sigma_{k(1)},\sigma_{k(2)},\ldots,\sigma_{k(s)}$,
along with the corresponding tropical roots of $\min \, \ell(t)$.
We call this list the \emph{covering certificate}.

We next present an  algorithm that decides
whether a given non-degenerate tropical line lies on a 
given tropical cubic surface.
It makes five calls to the Covering Subroutine.
An illustration of Algorithm \ref{algo:line-on-hypersurface}
is given in Example~\ref{ex:honeycomb:motif3D}.

\begin{algorithm}[htb]
  \caption{Deciding if a tropical line $L$ lies on a tropical surface $S$ in $\RR^3$}
  \label{algo:line-on-hypersurface}
  \begin{algorithmic}[1]
    \Require{The tropical Pl\"ucker vector $P$ for $L$, and a tropical polynomial $F$ that defines $S$.}
    \Ensure{Either a certificate that $L$ lies in $S$, or a point in the set difference  $L\backslash S$.}
            \State Determine the labeled type $ij|kl$ of $L$
 \State Compute the vertices $q_{ij}$ and $q_{kl}$ of $L$ via the formulas in \eqref{eq:trop-line:vertices}
    \State         Find parametrizations for the bounded edge and the four rays of $L$. 
    These are linear maps:
$ \,   [0,1] \rightarrow [q_{ij}, q_{kl}] \,$ and
    $\,[0,\infty) \rightarrow q_{ij}+\RR_{\geq 0}\cdot\omega_i\,$ and  $\,\cdots \,$ and
   $\,[0,\infty) \rightarrow q_{kl}+\RR_{\geq 0}\cdot\omega_l$.
   \For{each of the five linear maps above}
   \State Substitute the map into $F$. Get  an interval $U$ and a list $\ell(t)$ of linear polynomials.
   \State Apply the Covering Subroutine to $(U,\ell(t))$ and obtain the answer yes or no.
   \If{no} obtain $u \in U$, plug into linear map, and output  resulting point in $L \backslash S$.
   \EndIf
   \If{yes} obtain the covering certificate $(\sigma_{k(1)},\sigma_{k(2)},\ldots,\sigma_{k(s)})$ and save it.
   \EndIf
   \EndFor 
   \If{all five answers were yes}
   \State Output the covering certificates for the bounded edge and the four rays of $L$. 
   \label{lino:return}
   \EndIf
  \end{algorithmic}
\end{algorithm}

\begin{example}\label{ex:honeycomb:motif3D}
Fix the line $L$ with
  $P = (26, 6, 17, 7, 18, 0)$ and the cubic $F$ with
  \begin{equation}\label{eq:honeycomb:coefficients} % not generic!
    C \,=\, (32, 17, 20, 41, 26, 17, 32, 33, 36, 54, 8, 1, 14, 4, 7, 18, 0, 0, 0, 0).
  \end{equation}
  This vector induces the honeycomb triangulation $\#\idHoneycomb$ from \cite[\S6]{PV}:
  \begin{small}
    \[
      \begin{array}{l}
      \{0,1,4,10\}, \,  \{1,2,5,11\}, \, \{1,4,5,13\},  \, \{1,4,10,13\}, \, \{1,5,11,13\}, \, \{1,10,11,13\}, \\ 
      \{2,3,6,12\}, \,  \{2,5,6,14\}, \, \{2,5,11,14\}, \,  \{2,6,12,14\},  \, \{2,11,12,14\} \,  \{4,5,7,13\}, \\
       \{5,6,8,14\}, \, \{5,7,8,15\}, \, \{5,7,13,15\}, \, \{5,8,14,15\}, \{5,11,13,14\} , \{5,13,14,15\}, \\
       \{7,8,9,15\}, \{10,11,13,16\}, \{11,12,14,17\}, \, \{11,13,14,18\}, \,  \{11,13,16,18\}, \\
         \{11,14,17,18\}, \,  \{11,16,17,18\},\, \{13,14,15,18\}, \,  \{16,17,18,19\}. \\
      \end{array}
    \]
  \end{small}%
 The tropical line $L$ is non-degenerate and of 
 labeled type $01|23$ because $P_{02}{+}P_{13}=P_{03}{+}P_{12}=24<26=P_{01}{+}P_{23}$.
    Using \eqref{eq:trop-line:vertices} we find $q_{01}=(19,20,0,11)$ and $q_{23}=(17,18,0,11)$.
  In all five iterations through steps 4--11, the answer is yes.
The covering certificates $\sigma$ are:
  \begin{equation}
  \label{eq:CC}
    \begin{array}{lcl}
    [q_{01},q_{23}] & \text{has} & s=1 \ \text{and} \ \sigma = (\{14,15\}) \\  
      q_{01}+\RR_{\geq0}\omega_0 & \text{has} & s=1 \ \text{and} \ \sigma = (\{14,15\}) \\
      q_{01}+\RR_{\geq0}\omega_1 & \text{has} & s=2 \ \text{and} \ \sigma = (\{14 ,15 \} ,\{5,8\}) \\
      q_{23}+\RR_{\geq0}\omega_2 & \text{has} & s=2 \ \text{and} \ \sigma = (\{ 14, 18\},\{11,17\}) \\
      q_{23}+\RR_{\geq0}\omega_3 & \text{has} & s=1 \ \text{and} \ \sigma = (\{15,18\}) \\
    \end{array}
  \end{equation}
  There are two special points where $\min \ell(t)$ is
  attained four times. At the point $q_{23}$, the minimum is attained thrice.
The relevant index sets are cells in the triangulation: two tetrahedra
 $\{5,8,14,15\}$ and $\{11,14,17,18\}$, and the triangle $\{14,15,18\}$.
 These data identify an occurrence of the motif 3D in
Table \ref{Table:dualmotif}.
\end{example}

\begin{remark}
  Algorithm~\ref{algo:line-on-hypersurface} can be turned into a method for identifying all non-degenerate tropical lines in a given tropical surface $S$ in $\RR^3$.
  Here is an alternative method for the same task.
  Let $F$ be the tropical polynomial defining $S$.
  First we compute the  \emph{dome} $\smallSetOf{(x,y)}{x\!\in\!\RR^3,\, y\!\leq\!F(x)}$.
  This is an unbounded polyhedron in $\RR^4$ which represents $F$.
  We obtain a description of the surface $S$ as a polyhedral complex by projecting the codimension 2 skeleton of the dome.
  The maximal cells of $S$ are obtained by a convex hull computation \cite[\S3]{HampeJoswig:2017}.
  From this we enumerate the poset of all cells of $S$; cf.\ \cite[Algorithm 1]{HampeJoswigSchroeter:MEGA2017}.
  Each pair of cells is a candidate for possible locations of the two vertices $q_{ij}$ and $q_{kl}$.
  These points are described as convex combinations of the cells' vertices with unknown coefficients.
  Whether or not they form the two vertices of a tropical line in $S$ can be decided by checking the feasibility of a linear program.

  Simon Hampe implemented a similar approach for tropical cubic surfaces.
  This is the function \ \texttt{lines\_in\_cubic} \ in the \polymake extension \atint~\cite{atint}, which
  is slightly different from  our Algorithm~\ref{algo:line-on-hypersurface}.
  First, \texttt{lines\_in\_cubic} also computes degerate lines; second, that function is tailored to the cubic case.
\end{remark}

\section{Motifs and their Occurrences}
\label{sec:motifs}

We now turn to the ten motifs in Table~\ref{Table:dualmotif}.
We are interested in their occurrences in the $14\,373\,645$ unimodular regular triangulations of $3\Delta_3$.
As before, our goal is the complete classification of all possibilities.
We begin by stating our main result. The proof is given by
exhaustive computations using Algorithm~\ref{algo:motif-occurrences}.

\begin{theorem} \label{thm:114}
  The number of occurrences of all motifs in the unimodular regular triangulations of $3\Delta_3$ varies between $27$ and $128$, as shown in Figure~\ref{fig:histogram}.
  There are no triangulations with precisely $122$, $124$, $125$ or $127$ occurrences.
\end{theorem}

\begin{figure}[h]
%   \begin{tikzpicture} \small
%     \begin{axis}[ybar, bar width=2pt, xtick={20,30,...,130}, minor x tick num=1, xmajorgrids=true, ymajorgrids=true]
%       \addplot coordinates {
% (27,1426) (28,1761) (29,4554) (30,16347) (31,25808) (32,54759) (33,96321) (34,148709) (35,204774) (36,282641) (37,360500) (38,432467) (39,515128) (40,584741) (41,638971) (42,694274) (43,722344) (44,735541) (45,745927) (46,734427) (47,707729) (48,682681) (49,642961) (50,593498) (51,552007) (52,503125) (53,452684) (54,408846) (55,364315) (56,321925) (57,286360) (58,250765) (59,218730) (60,191241) (61,165673) (62,142927) (63,123662) (64,105853) (65,90457) (66,77801) (67,66216) (68,56193) (69,48487) (70,40900) (71,34722) (72,30206) (73,25694) (74,22313) (75,19405) (76,16874) (77,15241) (78,13304) (79,11918) (80,10996) (81,9799) (82,8871) (83,8216) (84,7422) (85,6228) (86,5825) (87,5136) (88,4230) (89,3628) (90,3280) (91,2544) (92,2311) (93,1759) (94,1485) (95,1210) (96,1037) (97,679) (98,624) (99,451) (100,382) (101,265) (102,239) (103,166) (104,145) (105,125) (106,119) (107,66) (108,72) (109,48) (110,35) (111,36) (112,14) (113,19) (114,12) (115,9) (116,6) (117,6) (118,2) (119,3) (120,4) (121,1) (123,1) (126,4) (128,2) 
%       };
%     \end{axis}   
%   \end{tikzpicture}
%   \hfill
  \begin{tikzpicture} 
    \begin{axis}[ybar, bar width=2pt, xtick={20,30,...,130}, minor x tick num=1, xmajorgrids=true, ymajorgrids=true]
      \addplot coordinates {
(27,34096) (28,42264) (29,109296) (30,391688) (31,619188) (32,1312248) (33,2310868) (34,3567360) (35,4913688) (36,6780392) (37,8651136) (38,10376406) (39,12362068) (40,14029902) (41,15334176) (42,16656490) (43,17334840) (44,17647320) (45,17900688) (46,17619612) (47,16983852) (48,16378726) (49,15429624) (50,14237220) (51,13246684) (52,12068829) (53,10863180) (54,9806590) (55,8742456) (56,7721082) (57,6871684) (58,6013338) (59,5248800) (60,4585311) (61,3975540) (62,3426282) (63,2967344) (64,2536908) (65,2170584) (66,1864376) (67,1588716) (68,1346382) (69,1163388) (70,979260) (71,832968) (72,722880) (73,616344) (74,533904) (75,465584) (76,403500) (77,365568) (78,317746) (79,285840) (80,262500) (81,235116) (82,211362) (83,197124) (84,176768) (85,149400) (86,138690) (87,123228) (88,100572) (89,87036) (90,77994) (91,61032) (92,54978) (93,42204) (94,35124) (95,28968) (96,24444) (97,16272) (98,14616) (99,10824) (100,8808) (101,6360) (102,5532) (103,3984) (104,3384) (105,2988) (106,2760) (107,1548) (108,1572) (109,1152) (110,708) (111,840) (112,288) (113,456) (114,228) (115,216) (116,132) (117,144) (118,48) (119,72) (120,84) (121,24) (123,24) (126,32) (128,15) 
% total=14373645 totalorbitsum=344843867
      };
    \end{axis}   
  \end{tikzpicture}
  \caption{The distributions of the total number of motifs, counting triangulations.
    The highest frequency is 45 motifs, which occur in $17\,900\,688$ triangulations.
    These form $745\,927$ orbits, which is
           5.12\% of orbits of regular unimodular triangulations of $3\Delta_3$.
    The minimum at 27 is attained by $34\,096$ triangulations in $1426$ orbits, and the maximum at 128 by $15$ triangulations in two orbits.
  }
  \label{fig:histogram}
\end{figure}
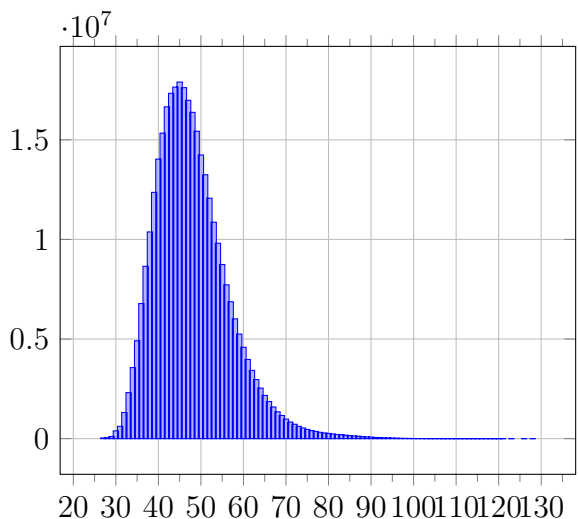

We now define the notion of occurrence.
Fix a regular unimodular triangulation $T$ of $3\Delta_3$.
Let $\motif$ be a motif, viewed as a labeled simplicial complex.
An \emph{occurrence} of $\motif$ in $T$ is a simplicial map from $\motif$ to $T$ 
that satisfies the conditions in the third column of Table~\ref{Table:dualmotif}.
These conditions include a bijection between the set $\{i,j,k,l\}$ of exits and the four facets of $3\Delta_3$.
Such a \emph{simplicial map} sends vertices of $\motif$ to vertices of $T$, while faces are mapped to faces.
Often occurrences are embeddings, but it can happen that two vertices of $\motif$ are mapped to the same vertex of $T$.
We shall see this in Example~\ref{ex:motif-no-embedding}.

An occurrence of a motif $\motif$ in $T$ is a map of simplicial complexes.
The definition above is subtle.
One might think that such a map is determined by the
image of the set of vertices of $\motif$.
This is not true! The same subcomplex 
 of $T$ may support several occurrences of a motif. We now present an example.

\begin{figure}[ht]
\centering
\includegraphics[width=\textwidth]{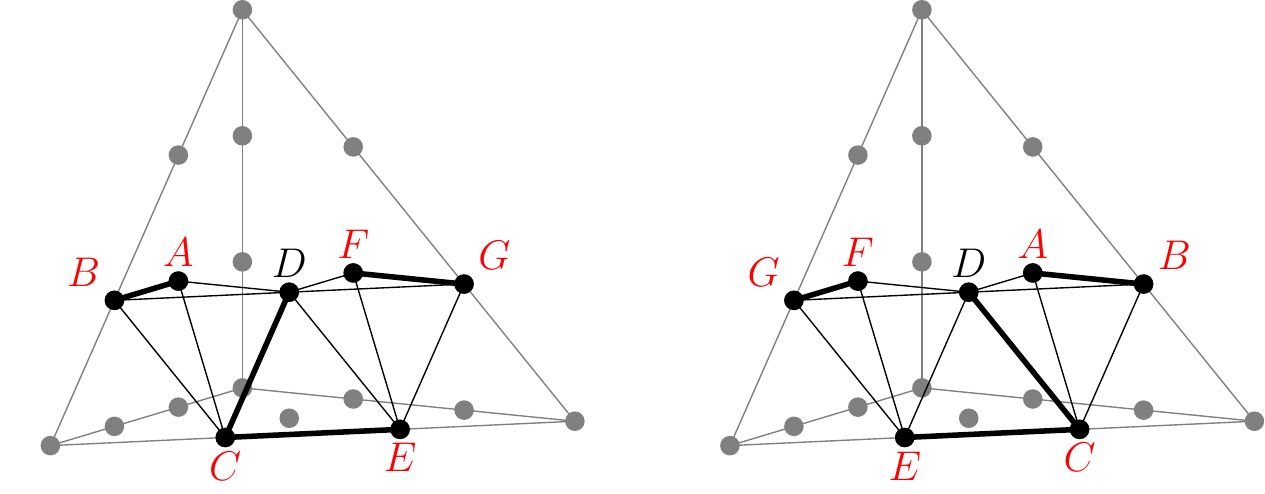}
\caption{Two distinct 3D motifs in the honeycomb triangulation that are 
supported by the same set of vertices (cf.~Example \ref{ex:honeycomb:alternate3D}).
Exit edges are marked.
 \label{fig:honeycomb:motif3D}}
\end{figure}

\begin{example}\label{ex:honeycomb:alternate3D}
The line $L$ in Example \ref{ex:honeycomb:motif3D}
gives an occurrence of the motif 3D in the honeycomb triangulation.
The corresponding simplicial map is given by 
\begin{equation}\label{eq:honeycomb:motif3D}
  \begin{matrix}
    A=11, \ B=17, \ C=18, \ D=14, \ E=15, \ F=5, \ G = 8, \\ i=3, \,\, j=2, \,\, k=0, \,\, l=1.
  \end{matrix}      
\end{equation}
This uses our fixed ordering of the lattice points in $3\Delta_3$, so the vertices are
\begin{small}
  \[
    A=(1101) ,\,\,  B=(0201) ,\,\,  C=(0210)  ,\,\, D=(0111) , \,\,  E=(0120)  ,\,\, F=(1011)  ,\,\, G=(0021) .
  \]
\end{small}
The left diagram in Figure~\ref{fig:honeycomb:motif3D} helps in
verifying the conditions from Table~\ref{Table:dualmotif}:
\[
  CE \subset F_3 ,\ AB \subset F_2 ,\ DE \subset F_0 ,\ FG \subset F_1 , \
  CD \subset \{x_2=1\} ,\ DE \subset \{ x_1=1 \}.
\]
The motif \eqref{eq:honeycomb:motif3D}  is made visible in 
Example \ref{ex:honeycomb:motif3D}
by the line $L$ in the surface $S$.
The motif occurrence is seen in
the covering certificates (\ref{eq:CC}) given by
Algorithm~\ref{algo:line-on-hypersurface}.

The above occurrence is special in that the exit edge $\{15,18\}$ lies in the edge
$F_0\cap F_3$ of  $3\Delta_3$.
We can relabel the  points and the exits as follows:
\[
  \begin{matrix}
    A=5, \ B=8, \ C=15, \ D=14, \ E=18, \ F=11, \ G = 17, \\ \ i=3, \,\, j=1, \,\, k=0, \,\, l=2.
  \end{matrix}
\]
This is another occurrence of a 3D motif in $T$,
shown on the right in Figure~\ref{fig:honeycomb:motif3D}.

In conclusion, the same subcomplex of the honeycomb triangulation 
supports two distinct occurrences of the motif 3D.
However, it is impossible for both to be visible in the same
cubic surface. 
% For instance, consider the  cubic $S$ seen in
% Example \ref{ex:honeycomb:motif3D},
% with coefficients~\eqref{eq:honeycomb:coefficients}.
% The left occurrence in Figure \ref{fig:honeycomb:motif3D} is visible in $S$:
% it is a line $L$ with Pl\"ucker coordinates $P$. But
% the right occurrence is not visible: there is no matching tropical line in $S$.
To ascertain whether an occurrence of a motif is visible in a specific cubic surface is our problem in Section~\ref{sec:schlaefli}.
\end{example}

We now show all motif occurrences in a given triangulation. As it stands,
 Algorithm~\ref{algo:motif-occurrences}  is too na\"{\i}ve to be useful.
The number of vertices of a motif varies between four (type 3I) and eight (type 3F).
For the 3F motif alone we would need to enumerate and check $20^8 = 25.6\cdot 10^9$ 
potential simplicial maps into~$T$.

\begin{algorithm}[h]
  \caption{Finding all motif occurrences}
  \label{algo:motif-occurrences}
  \begin{algorithmic}[1]
    \Require{Unimodular regular triangulation $T$ of $3\Delta_3$.}
    \Ensure{The list of all occurrences of motifs in $T$.}
    \ForEach{motif $M$}
    \ForEach{map $\motif$ from the vertices of $M$ to the 20 lattice points in $3\Delta_3$}
    \If{$\motif$ is simplicial into $T$ and the conditions in Table~\ref{Table:dualmotif} are satisfied}
    \State output $\motif$.
    \EndIf
    \EndFor
    \EndFor
  \end{algorithmic}
\end{algorithm}

In practice, it is  essential to exploit symmetries and other simplifications.
A \emph{symmetry} of a motif $\motif$ is a simplicial bijection from the labeled simplicial complex $\motif$ to itself such that the conditions in the third column of Table~\ref{Table:dualmotif} are preserved.
Two symmetric occurrences of a motif yield the same line in a given tropical surface (or none).
The symmetries of a motif form a group.
The following lemma is derived by direct inspection from
the data in Table~\ref{Table:dualmotif}.

\goodbreak

\begin{lemma}
  The ten motifs of tropical cubic surfaces have the following symmetry groups.
In each case,  generators $g_1, g_2, \ldots$ and a  description are given: \begin{small}
  \begin{itemize}
  \item[(3A)] $g_1 = (E \ F)$.
    Cyclic group of order $2$.
  \item[(3B)] $g_1 = (B \ C)(i \ j)$, $g_2=(A \ F)(B\ D)(C\ E)(i \ k)(j \ l)$. % the double transposition (D E)(k l) is redundant
    % gap> b:=Group([(2,3)(9,10), (4,5)(11,12), (1,6)(2,5)(3,4)(9,11)(10,12)]);
    % B C  i j     D E  k  l     A F  B D  C E  i k   j  l
    % gap> StructureDescription(b);
    % "D8"
    Dihedral group of order~$8$.
  \item[(3C)] $g_1=(B \ C)(i\ j)$, $g_2=(D\ E)$, $g_3=(F\ G)$.
    Elementary abelian group of order $8$.
  \item[(3D)] $g_1=(A\ B)$, $g_2=(F\ G)$.
    Elementary abelian group of order $4$.
  \item[(3E)]
    $g_1=(B\ C)(i\ j)$, $g_2=(D\ E)$, $g_3=(B\ C)(D\ F)(E\ G)(i\ j)(k\ l)$. % the transposition (F G) is redundant
    Nonabelian group~of order $16$: direct product of an order $2$ group $\langle g_1 \rangle$
        and a dihedral group $\langle g_2,g_3 \rangle$ of order~$8$.
    % gap> e:=Group([(2,3)(9,10), (4,5), (6,7), (2,3)(4,6)(5,7)(9,10)(11,12)]);
    % B C  i j     D E    F G    B C  D F  E G  i j   k  l
    % gap> StructureDescription(e);
    % "C2 x D8"
  \item[(3F)] $g_1=(A \,B)$, $g_2=(C\, D)$, $g_3=(E\, F)$, $g_4=(G\, H)$, 
    $g_5=(A\, H)(B\, G)(C\, F)(D\, E)(i\ k)(j\ l)$.
    Nonabelian group of order $32$.
    Here $g_1,g_2,g_3,g_4$ span an abelian subgroup of order~$16$.
    % gap> f:=Group([(1,2), (3,4), (5,6), (7,8), (1,8)(2,7)(3,6)(4,5)(9,11)(10,12)]);
    % A B    C D    E F    G H    A H  B G  C F  D E  i k   j  l
    % gap> StructureDescription(f);
    % "(C2 x C2 x C2 x C2) : C2"
  \item[(3G)] $g_1=(A\ B)$, $g_2=(C\ D)$, $g_3=(E\ F)$.
    Elementary abelian group of order~$8$.
  \item[(3H)] $g_1=(A\ B)$, $g_2=(C\ D)(k\ l)$.
    Elementary abelian group of order $4$.
  \item[(3I)] $g_1=(A\ B)$, $g_2=(C\ D)$.
    Elementary abelian group of order $4$.
  \item[(3J)] $g_1=(B\ C)$, $g_2 =(D\ E)$.
    Elementary abelian group of order $4$.
  \end{itemize}
  \end{small}
\end{lemma}

We next show that occurrences of motifs are generally not embeddings.

% $HJ=retrieve_by_id(5054117);
% $HJ->MOTIFS3F->[0]->properties();
% type: Motif
%
% TYPE
% F
%
% POINTS
% 15 18 11 17 14 15 2 9
%
% EXITS
% 2 3 0 1
%
% TETRAHEDRA
% 25 24 20

\begin{example}\label{ex:motif-no-embedding}
  The motif 3F occurs in the triangulation \eqref{eq:typical:triangulation} via the labeling
  \begin{small}
    \[
      A=15 , \,  B=18 ,\, C=11 ,\, D=17 ,\, E=14 ,\, F=15 ,\, G=2 ,\, H=9 , \ i=2 ,\, j=3 ,\, k=0 ,\, l=1 .
    \]
  \end{small}
In this occurrence,  $A$ and $F$ are mapped to the same point, labeled by~$15$.
\end{example}

To obtain Theorem~\ref{thm:114}, we developed a highly efficient version of 
Algorithm~\ref{algo:motif-occurrences}, we implemented it in \polymake, and
we applied it to millions of triangulations. This required substantial speed-ups,
based on structural constraints that control the combinatorial explosion.
In the rest of this section, we present a sample of such constraints, and we discuss how they are used.

\begin{lemma}\label{lem:motif3A}
  Vertex $A$ is distinct from $E$ and $F$ in any occurrence of motif 3A.
\end{lemma}

\begin{proof}
  If $A$ coincides with $E$ or $F$ then $A$ has coordinates $i$, $k$ and $l$ equal to zero.
  Moreover, the condition $AD \subseteq \{x_i+x_j=1\}$ implies that the coordinate $j$ is equal to one.
  This is impossible, since the four coordinates sum to three. 
\end{proof}

% It may happen that the vertex $B$ coincides with $E$ or $F$ in an occurrence of motif 3A.
% The next lemma means that it helps that we are considering only non-degenerate tropical lines.

% \begin{lemma}
%  Vertex $E$ is distinct from the vertices $A$ and $B$ in any occurrence of motif 3H.
% \end{lemma}

% \begin{proof}
%  Suppose that $E$ coincides with $A$ or $B$.
%  Then the vertex of the surface which is dual to the tetrahedron $ABCD$ lies on the edge dual to the triangle $CDE$.
%  Therefore the line defined by the motif is degenerate.  This is not allowed in our definition of
%  occurrence.
% \end{proof}

Our strategy for enumerating motif occurrences is to find the possible ways in which the 
simplices of a motif are mapped into the given triangulation.
This leads to more book-keeping in Algorithm~\ref{algo:motif-occurrences}, 
to be used for shortcuts.
% Before we reveal more details, let us estimate some orders of magnitude.
% For instance, the 3A motif has two triangles (of which there are 72 in any unimodular triangulation of $3\Delta_3$) and one tetrahedron (of which there are 27), resulting in $139 \, 968$ combinations to consider.
% Yet we also need to take the choices for the exits into account. This contributes another multiplicative factor of $4!=24$, yielding $3\, 359 \,232$ cases as the grand total.
% So this is not directly viable either.
We exploit the following features in the various motifs.
A tetrahedron $T$ is called \emph{sided} if it has one edge on a
facet $F_i$ of $3\Delta_3$ and the opposite edge lies on the plane $x_i=1$. The associated tropical line  
contains the vertex of the surface dual to $T$ in the interior of the ray in direction $\omega_i$. 
We call $T$ \emph{split} if it has two opposite edges with prescribed exits. Here there are two possibilities. The line has two adjacent rays 
in directions given by the exits, and one ray contains in its interior the vertex dual to the split tetrahedron. Or the bounded edge contains the vertex dual to the split tetrahedron in its interior, and the rays in directions given by the exits are not adjacent.  We say that $T$ is \emph{centered} if the constraints in Table~\ref{Table:dualmotif} 
 induce a bijection between its vertices and the facets of $3\Delta_3$. Its dual vertex lies in the interior of the bounded edge of the tropical line. Finally, a triangle in a motif is \emph{dangling} if it has two edges with required exits. The tropical line has a vertex in the interior of an edge of the surface. The two rays adjacent to that vertex have direction given by the exits of the dangling triangle. 

The features we defined above occur in the ten motifs as follows:
\begin{itemize}
\item The following tetrahedra are sided: $CDEF$ in motif 3A, $DEFG$ in 3C, $BCDE$ and $BCFG$ in 3E, $ABCD$ and $EFGH$ in 3F, $CDEF$ in 3G, $ABCD$ and $ABCE$ in 3J.
\item Tetrahedron $DEFG$ in 3D is split; so are $CDEF$ in 3F and $CDEF$ in 3G.
\item Tetrahedron $BCDE$ in motif 3B is centered.
\item Triangle $ABD$ in motif 3A is dangling, likewise $ABC$ and $DEF$ in 3B, $ABC$ in 3C, $CDE$ in 3D, $ABC$ in 3E, and $CDE$ in 3H.
\end{itemize}

% Sided tetrahedra yield a further simplification for identifying occurrences:

% \begin{lemma}
%  Let $\tau$ be a sided tetrahedron in an occurrence of a motif with opposite edges $e$ and $e'$ respectively contained in $F_l$ and $\{x_l=1\}$.
%  Then the index $l$ is uniquely determined, i.e., there is no index 
%  $\,m \not =l$ such that $e$ is also contained in $F_m$ and $e'$ in $\{x_m=1\}$.
% \end{lemma}

% \begin{proof}
%  Suppose that $e$ is contained in two faces $F_l$ and $F_m$.
%  Then $e$ lies on an edge of $3 \Delta_3$, and
%  the edge $e'$ lies on the line $\{x_l=1\} \cap \{x_m=1\}$.
%  This implies that $e$ and $e'$ are parallel.
%  However, this is impossible since they form opposite edges in a tetrahedron. 
% \end{proof}

Our strategy for Theorem \ref{thm:114}
is to first enumerate the features of a  triangulation, i.e.~its 
sided, split and centered tetrahedra, and its dangling triangles.
This is combined with searching for 
% larger and larger fragments of 
occurrences of a motif by local extensions.

We illustrate this for the 3A motif with a heuristic estimate for the number of subcases arising.
Let $T$ be the triangulation in \eqref{eq:typical:triangulation} and in Example~\ref{ex:typical:all-motifs} below.
We start out by finding the candidates for the sided tetrahedron $CDEF$, 
with the exit $EF$ on facet $F_l$.
Considering all labelings,  there are $114$ choices for this in $T$.
Next we need to find the candidates for $A$.
Here it suffices to consider those which are in the link of the edge $CD$.
For instance, the link for $\{C,D\} = \{15,11\}$ has six vertices. The number six
appears to be typical and we use this number for our estimate.
By Lemma~\ref{lem:motif3A}, $A$ must be distinct from $E$ and $F$, reducing the number of candidates to four. 
We further exclude any $A$ where $AC$ does not lie in the boundary of $3\Delta_3$.
For the remaining ones we try the three directions other than $l$, which is already fixed.
The only item missing is the vertex $B$.
Assuming, e.g., $A=18$ we need to check three candidates in the link of $AD$ (four minus one for $C$, because $A\neq C$) and two remaining exits.
This leads to $114 \cdot 4 \cdot 3 \cdot 3 \cdot 2 = 8208$ cases, including all possible labelings.
In fact, the enumeration  is even faster, as many of these cases can be ruled out early
while the various conditions in Table~\ref{Table:dualmotif} are being checked.
Summing up, the number of subcases considered by this approach is much smaller than 3.3 million subcases for one 3A motif one sees in a na\"{\i}ve backtracking search.

% fan > tets_and_exits($T5054117,1);
% tets_and_exits stats: #boundary_edge_star=54 #sided_tets=15 #split_tets=114 #centered_tets=63 #special_triangles=31

% two example motifs of type 3A:
%
% POINTS:     18 17 15 11 1 9
% EXITS:      0 2 3 1
% TETRAHEDRA: 18
%
% POINTS:     18 19 15 11 1 9
% EXITS:      3 2 0 1
% TETRAHEDRA: 18

% fan > tets_and_exits($H,1);
% tets_and_exits stats: #boundary_edge_star=54 #sided_tets=24 #split_tets=104 #centered_tets=3 #special_triangles=44

\section{Schl\"afli Cones}
\label{sec:schlaefli}

In Section \ref{sec:motifs} we studied the occurrences of  motifs in the $\numberOfOrbits$ types of regular unimodular triangulations of $3 \Delta_3$.
Their number per type ranges between $27$ and $128$. 
In this section we focus on individual smooth tropical cubic surfaces from a fixed secondary cone $\seccone(T)$.
Every tropical line on a generic surface gives a motif that occurs in $T$.
But the converse is not true. An occurrence of a motif need not contribute a tropical line to a given surface.

Let $T$ be a regular unimodular triangulation of $3\Delta_3$.
Each point $C$ in the  open secondary cone $\seccone(T)$ specifies a smooth tropical cubic surface $S_C$
which is dual to the triangulation $T$.
Given an occurrence $\motif$ of a motif in  $T$, we say that $\motif$ is \emph{visible in  $S_C$} if there is a tropical line $L$ in $S_C$ that has the dual complex $\motif$. 
We write $\mathcal{M}_C$ for the set of all motifs that are visible in  $S_C$.

We regard two vectors $C$ and $C'$ in $\seccone(T)$ as \emph{equivalent} if $\mathcal{M}_C = \mathcal{M}_{C'}$.
Each equivalence class is a finite union of relatively open convex polyhedral cones in $ \RR^{20}$.
The full-dimensional cones among these are the \emph{Schl\"afli cones}.
Each facet of a Schl\"afli cone is defined by a linear form in $\ZZ[c_0,c_1,\ldots,c_{19}]$.
This linear form is unique up to scaling.
We identify this linear form with the hyperplane it defines, and we call it a \emph{Schl\"afli wall} for the type~$T$.
The collection of all Schl\"afli walls defines a hyperplane arrangement in $\RR^{20}$.

The \emph{Schl\"afli fan} of the combinatorial type $T$ is the subdivision of $\seccone(T)$ induced by the Schläfli walls of type $T$.
Every maximal cone of the Schl\"afli fan is fully contained in a Schl\"afli cone.
Hence, the set $\mathcal{M}_C$ is constant for all surfaces~$S_C$ in a fixed maximal cone of the Schl\"afli fan.
A tropical cubic surface~$S_C$ is \emph{generic} if its coefficient vector $C$ is in the interior of a Schl\"afli~cone. 
Notice that, in general, a Schl\"afli cones need not be a cone of the Schl\"afli fan.

There are $\numberOfOrbits$ distinct Schl\"afli fans. 
Algorithm \ref{algo:visibility} finds their Schl\"afli walls. 
We coded this in \macaulay \cite{M2}.
Here is one of the results we found:

\begin{theorem} \label{thm:1426}
  For each of the $1426$ types in Theorem \ref{thm:114}
   with exactly $27$ motifs, the secondary cone remains undivided in the  Schl\"afli fan.
  Among these, $1396$ types feature isolated tropical lines only.
The remaining $30$ have precisely one occurrence of motif 3I; in particular, motif 3J does not occur at all.
\end{theorem}

The situation is different for many triangulations $T$ with more than $27$ motif occurrences.
Here, the Schl\"afli fan is nontrivial; it  does divide $\seccone(T)$
into smaller cones, according to which tropical lines lie on the various cubic surfaces.
Each Schl\"afli wall arises (non-uniquely) from some
motif $\motif$ that occurs in~$T$. If one crosses
from one Schl\"afli cone to a neighboring one through
a shared facet, then the set $\mathcal{M}_C$ 
of visible motifs changes.
If a motif $\motif$ is no longer visible then the Schl\"afli wall gives a linear inequality that
is necessary for $\motif$ to be visible. We write
$\mathcal{W}_\motif$ for the set of Schl\"afli walls arising from the motif~$\motif$.

\begin{lemma} \label{lemma:evident} Let $\motif$ be an occurrence of a motif 3F, 3G or 3I in a type $T$.
Then $\mathcal{W}_\motif = \emptyset$. In other words,
$\motif$ is visible in every tropical cubic surface of  type $T$.
\end{lemma}

\begin{proof} This was shown for the motifs 3G and 3I in \cite[Proposition 23]{PV}.
Now consider the motif 3F.
Suppose that $\motif$ is an occurrence  of 3F. The three tetrahedra
  $ABCD$, $CDEF$ and $EFGH$ are
  dual to three vertices  of $S_C$. The necessary conditions on the  edges in  
  Table~\ref{Table:dualmotif}  allow trespassing segments 
  respectively in the directions $\omega_j$, $\omega_i+\omega_j$ and $\omega_l$. 
  Thanks to the exits of the three tetrahedra, these segments can 
  always be completed to a   tropical line, irrespective of the specific values of the 
  parameters $c_i$.
\end{proof}

We now discuss how the set of walls $\mathcal{W}_\motif$ can be computed
for the other motifs. The basic idea is this. Given $\motif$, we compute a 
tropical line that matches the combinatorics in~$\motif$.
The line is uniquely determined by its two vertices. Their coordinates
 are linear forms in $C=(c_0,c_1,\ldots,c_{19})$. 
 In this section we use lowercase letters $c_i$ instead of uppercase letters $C_i$
 for the coordinates of the tropical coefficient vector $C$, 
 so as to make our tables more readable.
  
\begin{algorithm}[h] 
  \caption{Computing the visibility cone and Schl\"afli walls of a motif}
  \label{algo:visibility}
  \begin{algorithmic}[1]
    \Require{Secondary cone of a unimodular regular triangulation $T$ of $3\Delta_3$, 
    and an occurrence $\motif$ of a motif in $T$}
    \Ensure{Visibility cone and Schl\"afli walls of $\motif$}
    \State Compute the vertices of the tropical line $L$ dual to $\motif$.
       \ForEach{vertex $V$ of $L$}
       \State Substitute the coordinates of $V$ in the $20$ monomials of the tropical cubic polynomial.
       \State Get linear inequalities by requiring that $V$ lies on the prescribed cell of the surface.
       \EndFor
      \State Construct the cone defined by these linear inequalities.
      \State Compute the visibility cone by intersecting that cone with the secondary cone.
      \State Remove redundant facets.
       \State Output the visibility cone and Schl\"afli walls.
  \end{algorithmic}
\end{algorithm}

  If the $c_i$ take on values in $\RR$ then the
   tropical line may or may not be contained in $S_C$. 
We require that it lies in $S_C$ as prescribed by $\motif$.
Each vertex must lie on a cell of $S_C$ that is specified by  the tropical cubic polynomial.
These linear forms must be equal and bounded above by the
other ones. We consider these linear inequalities together with those that define the secondary cone.  They define the  \emph{visibility cone} of $\motif$ in $\seccone(T)$.  The irredundant linear inequalities for this cone give us the linear forms in $\mathcal{W}_\motif$.  A Schl\"afli cone is the intersection of the full dimensional visibility cones of the
visible motifs.

If all  linear inequalities we found are redundant, then the  visibility cone  equals the secondary 
cone.  In that case, the motif is visible in each surface $S_C$ with $C \in \seccone(T)$,
and  the motif   is \emph{globally visible}.  This holds in
 Theorem \ref{thm:1426}.

\begin{table}[tbh]\centering
  \caption{The triangulation $\#\idHampeJoswigTriangulation$ from (\ref{eq:typical:triangulation}) has $24$ globally visible motifs.}
  \label{tab:ex-motifs:global}
\begin{tiny} 
\begin{tabular*}{.75\linewidth}{@{\extracolsep{\fill}}ccc@{}}
  \toprule
  Index & Points & Exits   \\
  \midrule
  \multicolumn{3}{c}{Motifs 3B}  \\
  \midrule
  0&9, 15, 7, 1, 18, 19 & 0, 1, 2, 3   \\
  \midrule
  \multicolumn{3}{c}{Motifs 3D} \\
  \midrule
  1& 9, 15, 2, 11, 1, 9, 15 & 1, 0, 2, 3   \\
  2& 3, 14, 2, 11, 1, 15, 18 &  1, 0, 2, 3 \\
  3& 9, 15, 2, 11, 1, 15, 18 &  1, 0, 2, 3  \\
  4& 14, 15, 2, 11, 1, 15, 18 & 1, 0, 2, 3 \\ 
  5& 3, 14, 2, 11, 1, 18, 19 & 1, 0, 2, 3  \\ 
  6& 9, 15, 2, 11, 1, 18, 19 & 1, 0, 2, 3 \\ 
  7&14, 15, 2, 11, 1, 18, 19 & 1, 0, 2, 3   \\ 
  8&9, 15, 1, 11, 2, 3, 14 & 1, 3, 2, 0   \\  
  9&9, 15, 1, 11, 2, 14, 15 & 1, 3, 2, 0   \\ 
  10& 9, 15, 1, 11, 2, 9, 15 &  1, 3, 2, 0    \\  
  11& 2, 3, 14, 11, 17, 15, 18 & 0, 1, 2, 3   \\ 
  12& 2, 3, 14, 11, 17, 18, 19 & 0, 1, 2,  3   \\
  \midrule
   \multicolumn{3}{c}{Motifs 3F}  \\
  \midrule
  13& 15, 18, 11, 17, 14, 15, 2, 9 & 2, 3, 0, 1 \\
  14& 18, 19, 11, 17, 14, 15, 2, 9 & 2, 3, 0, 1\\
   \midrule
   \multicolumn{3}{c}{Motifs 3G}  \\
  \midrule
 15 & 9, 15, 2, 11, 3, 14 &1, 3, 2, 0 \\
16 & 9, 15, 2, 11, 14, 15 & 1, 3, 2, 0 \\
17 & 9, 15, 1, 11, 15, 18 & 0, 1, 2, 3 \\
18 & 9, 15, 1, 11, 18, 19 & 0, 1, 2, 3 \\
  \midrule
  \multicolumn{3}{c}{Motifs 3H}  \\
  \midrule
19&7, 15, 1, 18, 19 & 0,  1, 2, 3   \\
 20&9, 15, 2, 14, 3 &  1, 3, 2, 0  \\
    \midrule
  \multicolumn{3}{c}{Motifs 3I}  \\
  \midrule
21& 1, 11, 9, 15 & 1, 2, 0, 3\\
 22& 2, 11, 9, 15 & 1, 2, 0, 3 \\
    \midrule
  \multicolumn{3}{c}{Motifs 3J} \\
    \midrule
  23&11, 9, 15, 1, 2 & 0, 3, 1, 2  \\
  \bottomrule
\end{tabular*}
\end{tiny}
\end{table}

If  Algorithm \ref{algo:visibility} finds irredundant linear forms,
then we distinguish two cases, according to the dimension of the visibility cone. If the visibility cone is full dimensional, then the motif is \emph{partially visible}. Finally, a visibility cone might not be full dimensional. This  means that it is contained in a linear space of positive codimension.  A motif with visibility cone of  lower dimension is not visible in generic surfaces. We therefore call it \emph{hardly visible}. 

We now illustrate these concepts for the tropical cubic surface from \eqref{eq:typical:triangulation}.

\begin{table}[h]\centering
\caption{The triangulation $\#\idHampeJoswigTriangulation$ from (\ref{eq:typical:triangulation}) has $18$ partially visible motifs.}
  \label{tab:ex-motifs:partial}
\begin{tiny}
\begin{tabular*}{\linewidth}{@{\extracolsep{\fill}}cccc@{}}
\toprule
  Index & Points & Exits & Schl\"afli walls \\
 \midrule
  \multicolumn{4}{c}{Motifs 3A} \\
  \midrule
0&18, 17, 15, 11, 2, 9 & 0, 2, 3, 1 &  \begin{tabular}{c}$-c_2+c_9+c_{11}-c_{15}+c_{17}-c_{18}$, \\
 $ c_2-c_9-c_{11}+2c_{15}-c_{17}-c_{18}+c_{19} $ \end{tabular}\\ \addlinespace[0.5em]
1&18, 19, 15, 11, 2, 9 & 3, 2, 0, 1 & $-c_2+c_9+c_{11}-c_{15}+c_{17}-c_{18}$ \\ 
2&18, 19, 15, 11, 2, 9 & 0, 2, 3, 1 & $ -c_2+c_9+c_{11}-2c_{15}+c_{17}+c_{18}-c_{19} $ \\  \addlinespace[0.5em]
3&18, 17, 15, 11, 1, 9 & 0, 2, 3, 1&  \begin{tabular}{c}$
c_1-c_9-2c_{11}+2c_{15}+c_{17}-c_{18}$,\\ $
 -c_1+c_9+2c_{11}-c_{15}-c_{17}-c_{18}+c_{19}$ \end{tabular}\\  \addlinespace[0.5em]
4&18, 19, 15, 11, 1, 9 & 3, 2, 0, 1 &$c_1-c_9-2c_{11}+2c_{15}+c_{17}-c_{18} $\\ 
5&18, 19, 15, 11, 1, 9 & 0, 2, 3, 1 & $ c_1-c_9-2c_{11}+c_{15}+c_{17}+c_{18}-c_{19} $\\ 
  \midrule
  \multicolumn{4}{c}{Motifs 3B}  \\
  \midrule
6&17, 18, 11, 1, 15, 7 & 0, 2, 1, 3  & \begin{tabular}{c} $ c_1-c_7+c_9-c_{11}-c_{15}+c_{18}$, \\ $-c_1+2c_{11}+c_{15}-c_{17}-2c_{18}+c_{19} $ \end{tabular} \\  \addlinespace[0.5em]
7&17, 18, 11, 1, 15, 9 & 0, 2, 1, 3 & \begin{tabular}{c} $-c_1+c_7-c_9+c_{11}+c_{15}-c_{18}$, \\ $-c_1+2c_{11}+c_{15}-c_{17}-2c_{18}+c_{19} $\end{tabular} \\   \addlinespace[0.5em]
8&19, 18, 11, 1, 15, 7 & 0, 2, 1, 3 & \begin{tabular}{c} $ c_1-c_7+c_9-c_{11}-c_{15}+c_{18}$, \\ $c_1-2c_{11}-c_{15}+c_{17}+2c_{18}-c_{19}$ \end{tabular} \\  \addlinespace[0.5em]
9&19, 18, 11, 1, 15, 9 & 0, 2, 1, 3 & \begin{tabular}{c} $-c_1+c_7-c_9+c_{11}+c_{15}-c_{18}$, \\ $c_1-2c_{11}-c_{15}+c_{17}+2c_{18}-c_{19}$ \end{tabular} \\ 
  \midrule
 \multicolumn{4}{c}{Motifs 3D} \\
 \midrule
10 &1, 9, 15, 11, 17, 15, 18 & 0, 1, 2, 3  & $-c_1+c_9+2c_{11}-2c_{15}-c_{17}+c_{18}$ \\  
11 & 2, 9, 15, 11, 17, 15, 18 & 0, 1, 2, 3 & $ c_2-c_9-c_{11}+c_{15}-c_{17}+c_{18} $ \\ 
12& 1, 9, 15, 11, 17, 18, 19 & 0, 1, 2, 3 & $-c_1+c_9+2c_{11}-2c_{15}-c_{17}+c_{18}$ \\ 
13 &2, 9, 15, 11, 17, 18, 19 & 0, 1, 2, 3 & $c_2-c_9-c_{11}+c_{15}-c_{17}+c_{18}$ \\
  \midrule
 \multicolumn{4}{c}{Motifs 3H} \\
 \midrule
14&18, 19, 1, 13, 4 & 0, 2, 1, 3 &  $ -c_4+c_7+c_{13}-2c_{18}+c_{19} $ \\  
15&11, 18, 1, 15, 9 & 0, 2, 1, 3 &$ -c_1+c_7-c_9+c_{11}+c_{15}-c_{18} $ \\ 
16&18, 19, 1, 13, 7 & 0, 2, 1, 3 &$c_4-c_7-c_{13}+2c_{18}-c_{19}$ \\ 
17&11, 18, 1, 15, 7 &  0, 2, 1, 3 & $c_1-c_7+c_9-c_{11}-c_{15}+c_{18}$ \\ 
  \bottomrule
  
\end{tabular*}
\end{tiny}
\end{table}

\begin{table}
\caption{The triangulation $\#\idHampeJoswigTriangulation$ from (\ref{eq:typical:triangulation}) has $9$ hardly visible motifs.}
  \label{tab:ex-motifs:hardly}
\begin{tiny}
\resizebox{\textwidth}{!}{
\begin{tabular*}{1.0\linewidth}{@{\extracolsep{\fill}}cccc@{}}
\toprule
 Index & Points & Exits & Schl\"afli walls \\
 \midrule
 \multicolumn{4}{c}{Motifs 3D} \\
\midrule
0 &3, 14, 2, 11, 1, 9, 15 & 1, 0, 2, 3 &  \begin{tiny}\begin{tabular}{c}  $ c_{12}-2c_{14}+c_{15}$, \\ $c_1-3c_5+2c_9+c_{11}+c_{14}-2c_{15}$,  \\ $ c_9-2c_{12}+3c_{14}-3c_{15}+c_{17} $,  \\ $c_3-2c_6+c_8$, \ \ $-c_3+2c_6-c_8$ \\ Equations: $c_2-c_6-c_{11}+c_{14}$, \\  $2c_3-3c_6+c_9$,  $c_2+c_3-2c_6-c_{11}+c_{15}$  \end{tabular} \end{tiny}\\  \addlinespace[0.6em]
1& 14, 15, 2, 11, 1, 9, 15 & 1, 0, 2, 3 & \begin{tiny} \begin{tabular}{c}$c_1-3c_5+2c_9+c_{11}+c_{14}-2c_{15}$, \\  $ -c_3+c_9+c_{12}+c_{14}-2c_{15}$ \\ Equation: $c_2-c_9-c_{11}-c_{14}+2c_{15}$ \end{tabular} \end{tiny}\\   \addlinespace[0.6em]
2& 15, 18, 1, 11, 2, 3, 14 & 1, 3, 2, 0  &\begin{tiny} \begin{tabular}{c} $-2c_6+c_8+2c_{12}-c_{17}$, \ $c_8-c_9-c_{17}+c_{18}$,  \\  $c_1-3c_5+c_9+c_{11}+c_{12}-c_{17}$, \\ $c_9-c_{12}-c_{15}+2c_{17}-c_{18}$, \\ Equations:  \\ $c_2-c_3-c_{11}+c_{12}$,   $2c_2-c_3-2c_{11}+c_{17}$, \\ $2c_2-c_3-2c_{11}+c_{14}-c_{15}+c_{18}$ \end{tabular}\end{tiny} \\   \addlinespace[0.6em]
3& 18, 19, 1, 11, 2, 3, 14 & 1, 3, 2, 0  & \begin{tiny} \begin{tabular}{c} $-2c_6+c_8+2c_{12}-c_{17}$, \\ $c_1-3c_5+c_9+c_{11}+c_{12}-c_{17}$, \\ $c_8-c_9-c_{17}+c_{18}$, \ \ \ $-c_1+c_{11}+c_{13}-c_{18}$, \\ $c_9-c_{12}+2c_{17}-3c_{18}+c_{19}$, \\ $c_5-c_9-c_{11}+2c_{18}-c_{19}$, \\  Equations: $c_2-c_3-c_{11}+c_{12}$,  \\ $2c_2-c_3-c_7-2c_{11}+c_{13}+c_{14}$, \\ $2c_2-c_3-2c_{11}+c_{17}$, \ \  $c_7-c_{13}-c_{15}+c_{18}$, \\ $ 2c_7-2c_{13}-c_{15}+c_{19}$ \end{tabular} \end{tiny} \\   \addlinespace[0.6em]
4 & 15, 18, 1, 11, 2, 9, 15 & 1, 3, 2, 0 &  \begin{tiny} \begin{tabular}{c}$c_5-c_{11}-c_{15}+c_{18}$, \\ $c_1-3c_5+c_{11}+3c_{15}+c_{17}-3c_{18}$ \\ Equations:   $c_3-2c_6+c_8$, \\ $2c_3-3c_6+c_9$, $c_2-c_3-c_{11}+c_{12} $, \\ $c_2-c_6-c_{11}+c_{14}$, \\ $c_2+c_3-2c_6-c_{11}+c_{15}$, \\  $2c_2-c_3-2c_{11}+c_{17}$, \ \ $2c_2-c_6-2c_{11}+c_{18}$ \end{tabular} \end{tiny} \\    \addlinespace[0.6em]
5 & 18, 19, 1, 11, 2, 9, 15 & 1, 3, 2, 0 &\begin{tiny} \begin{tabular}{c} $-c_1+c_{11}+c_{13}-c_{18}$, \ \  $c_5-c_{11}-c_{18}+c_{19}$, \\ $c_1-3c_5+c_{11}+c_{17}+3c_{18}-3c_{19}$ \\
Equations: $ c_3-2c_6+c_8 $, \\ $ 2c_3-3c_6+c_9$, $c_2-c_3-c_{11}+c_{12}$, \\ $c_2-c_3+c_6-c_7-c_{11}+c_{13}$, \\  $c_2+c_3-2c_6-c_{11}+c_{15}$, \\  $c_2-c_6-c_{11}+c_{14} $,  \ $2c_2-c_3-2c_{11}+c_{17}$, \\ $2c_2-c_6-2c_{11}+c_{18}$, \ \  $3c_2-c_3-3c_{11}+c_{19}$ \end{tabular} \end{tiny} \\  \addlinespace[0.6em]
6 & 15, 18, 1, 11, 2, 14, 15 & 1, 3, 2, 0 & \begin{tiny} \begin{tabular}{c}$c_5-c_{11}-c_{15}+c_{18}$, \\ $c_1-3c_5+c_{11}+3c_{15}+c_{17}-3c_{18}$\\  Equations: $ c_3-2c_6+c_8$, \\  $2c_3-3c_6+c_9$, \ \ $c_2-c_3-c_{11}+c_{12}$, \\ $ c_2-c_6-c_{11}+c_{14}$, \ \  $c_2+c_3-2c_6-c_{11}+c_{15}$, \\ $2c_2-c_3-2c_{11}+c_{17}$, \ \ $2c_2-c_6-2c_{11}+c_{18}$ \end{tabular} \end{tiny} \\ \addlinespace[0.6em]
7 & 18, 19, 1, 11, 2, 14, 15 & 1, 3, 2, 0 & \begin{tiny} \begin{tabular}{c} $-c_1+c_{11}+c_{13}-c_{18}$, \ \ \ $c_5-c_{11}-c_{18}+c_{19}$ \\ $ c_1-3c_5+c_{11}+c_{17}+3c_{18}-3c_{19}$ \\ 
Equations: $c_3-2c_6+c_8$, \ $2c_3-3c_6+c_9$, \\ $ c_2-c_3+c_6-c_7-c_{11}+c_{13}$, \\ $c_2-c_3-c_{11}+c_{12}$, \ $c_2-c_6-c_{11}+c_{14}$, \\$c_2+c_3-2c_6-c_{11}+c_{15}$, \ \ $2c_2-c_3-2c_{11}+c_{17}$, \\ $ 2c_2-c_6-2c_{11}+c_{18}$, \ \ $3c_2-c_3-3c_{11}+c_{19}$ \end{tabular} \end{tiny} \\
\midrule
\multicolumn{4}{c}{Motifs 3H} \\
 \midrule
8 &9, 11, 1, 15, 7 & 0, 2, 1, 3 & \begin{tiny} \begin{tabular}{c} $ c_2-3c_5+c_7+c_9+c_{11}-c_{15}$, \\
Equation: $c_1-c_7-c_{11}+c_{15}$ \end{tabular} \end{tiny} \\ 
\bottomrule
\end{tabular*}
}
\end{tiny}
\end{table}

\begin{example}\label{ex:typical:all-motifs}
% $X = retrieve_by_id(5054117); print $X->N_MOTIF_TYPES, " = ", $X->N_ALL_MOTIFS;
% 6 5 0 24 0 2 4 7 2 1 = 51
  The triangulation $\#\idHampeJoswigTriangulation$ has $51$ occurrences of the motifs 3A, 3B, $\ldots \,$, 3J.
  Their frequencies are $6, 5, 0, 24, 0, 2, 4, 7, 2, 1 $.
  Lemma \ref{lemma:evident} says that the motifs 3F, 3G and 3I are globally visible.
  In Tables \ref{tab:ex-motifs:global}, \ref{tab:ex-motifs:partial} and \ref{tab:ex-motifs:hardly} we list all motifs, together with their sets of Schl\"afli walls $\mathcal{W}_\motif$.
  We describe how the Schl\"afli walls are computed for the motifs of type 3H.
  The motif $\motif$ consists of a tetrahedron $ABCD$ and a dangling triangle $CDE$.
  One of the vertices of the tropical line defined by $\motif$ is dual to the tetrahedron.
  In order for the line to be contained in the surface, the other vertex must lie on the edge dual to the dangling triangle, i.e., the minimum in the tropical polynomial must be achieved at the monomials corresponding to $C$, $D$ and $E$.
  These linear inequalities define the visibility cone.
  Note that the occurrence 8 of motif 3H in Table~\ref{tab:ex-motifs:hardly} is hardly visible, since its visibility cone is not full dimensional. 

  The list of partially visible motifs in Table~\ref{tab:ex-motifs:partial} shows that the Schl\"afli walls generate a hyperplane arrangement defined by the seven linear forms:
  \[
    \begin{array}{l}
      H_0: c_2 -c_9-c_{11}+c_{15}-c_{17}+c_{18} \\
      H_1: c_2-c_9 -c_{11}+2c_{15}-c_{17}-c_{18}+c_{19}\\
      H_2: c_1-c_9-2c_{11}+2c_{15}+c_{17}-c_{18}\\
      H_3: c_1-c_9-2c_{11}+c_{15}+c_{17}+c_{18}-c_{19}\\
      H_4: c_1-c_7+c_9-c_{11}-c_{15}+c_{18}\\
      H_5: c_1-2c_{11}-c_{15}+c_{17}+2c_{18}-c_{19} \\
      H_6: c_4-c_7-c_{13}+2c_{18} -c_{19} \\
    \end{array}
  \]
  We write $H_i^+$ and $H_i^-$ for  the two halfspaces defined by  these linear forms.

  Let us look at the Schl\"afli walls from partially visible motifs of type 3B.
  The hyperplanes for the Schl\"afli walls of these motifs are $H_4$ and $H_5$.
  They divide the secondary cone into  four cells $H_4^+ \, H_5^+$, $H_4^+ \, H_5^-$, $H_4^- \, H_5^+$, $H_4^- \, H_5^-$. 
  These four cells correspond in Table \ref{tab:ex-motifs:partial} to the occurrences 8, 6,  7 and 9, in this order.
  Each motif occurrence is visible in precisely that cell.

  For the motifs of type 3D, we also have two hyperplanes $H_0$ and $H_2$. 
  These give the Schl\"afli walls that divide the secondary cone into four cells.
  In the cells $H_0^+ \, H_2^+$  and $H_0^- \, H_2^-$ the  motifs  11, 13 and  10, 12  are visible, respectively.
  In the cell $H_0^- \, H_2^+$ none of the partially visible motif is visible.
  Finally, on the cell $H_0^+ \, H_2^-$ all the partially visible motifs are visible. 
  Moreover, when we pass through the Schl\"afli wall from $H_0^-$ to $H_0^+$, the motifs 0 and 1 of type 3A are no longer visible, while the motifs 11 and 13 of type 3D become visible. 
\end{example}

\begin{remark} In this section we have considered the problem whether a motif is visible on a certain tropical surface. If the motif is visible, the next natural step is to ask whether the tropical line 
defined by the motif is realizabile on a  cubic surface defined over a field with valuation. More precisely, given a tropical line $L$ on  a tropical surface $S$, we ask whether there exits a line $\ell$ on a cubic surface $\mathcal{S}$ such that $\textrm{trop}(\ell) = L$ and $\textrm{trop}(\mathcal{S}) = S$.  This realizability problem has been studied in \cite{BogartKatz} and \cite{BrugalleShaw}. The authors show that non-degenerate lines in  families of type 3I are not realizable on surfaces over a valued field  of characteristic zero. Moreover, in the recent article \cite{geiger} the result is extended to valued field with residue field of characteristic different from two.  The paper also provides an example of a line of type 3J which is realizable on a cubic surface defined over the field of $5$-adic numbers. 
\end{remark}
\section{The Universal Fano Variety and its Tropical Discriminant}
\label{sec:fano}

We now relate our combinatorial results to
classical algebraic geometry. The natural parameter space for our problem is
 the \emph{universal Fano variety}. Its points are pairs consisting of a line and 
 a cubic surface that contains it. The map onto the second factor is a $27$-to-$1$ 
cover of $\PP^{19}$. The fiber over a smooth cubic surface, regarded as a point in $\PP^{19}$,
is the \emph{Fano variety} on that surface, i.e.~the set of its $27$ lines.
The branch locus of the $27$-to-$1$ map is its \emph{discriminant},
a hypersurface in $\PP^{19}$. We shall see that the codimension one 
skeleton of the Schl\"afli fan plays the role of the tropical discriminant for this map.

We follow the approach to tropical geometry  in  the textbook 
 \cite{TropicalBook}. One starts with a classical variety, 
 defined by an ideal $I$ in a (Laurent) polynomial ring over a field with valuation.
  The  tropical variety $\Trop(I)$ is the set of all weight vectors $w$ whose 
  initial ideal $\initial_w (I)$ contains no monomials.
  We would like to apply this to the universal Fano variety 
for lines on cubic surfaces, represented by an ideal
in the polynomial ring in the unknowns $p_{ij}$ and $c_k$. This is
 the homogeneous coordinate ring of $\PP^{5} \times \PP^{19}$.
The first factor contains the Grassmannian $\Gr(2,4)$ of
 lines in $\PP^3$ as a quadratic hypersurface in~$\PP^5$.

 The quadric defining $\Gr(2,4)$ is the Pfaffian of the skew-symmetric matrix
 \[
   \cP \quad = \quad
\begin{small} \begin{pmatrix}
0 & p_{01} & p_{02} & p_{03} \, \\
-p_{01} &   0 & p_{12} & p_{13} \, \\
-p_{02} & -p_{12} &  0 & p_{23} \, \\
-p_{03} & -p_{13} & -p_{23} &  0 \, \end{pmatrix}\end{small} . \]
We have $\Pfaff(\cP) = p_{01} p_{23} - p_{02} p_{13} + p_{03} p_{12}$.
The line with Pl\"ucker coordinates $(p_{ij})$ is the image 
in $\PP^3$ of the column span of the associated rank $2$ matrix~$\cP$.

The second factor $\PP^{19}$  parametrizes cubic forms $f$. Its coordinates
are the coefficients $(c_0,c_1,\ldots,c_{19})$.
Fix a row vector of unknowns
$\lambda = (\lambda_0,\lambda_1,\lambda_2,\lambda_3)$ and
form the vector-matrix product $ \lambda \cP $.
 We write $f(\lambda \cP)$ for the polynomial 
obtained by replacing  $(w,x,y,z)$ with $\lambda \cP$.
Thus, $f(\lambda \cP)$ is a homogeneous cubic in $\lambda$.
Its $20$ coefficients are  bihomogeneous forms of degree $(3,1)$, like
\begin{equation}
\label{eq:mybihomo}
\begin{matrix}
p_{01} p_{12} p_{13} c_5- p_{12}^2 p_{13} c_8
+ p_{01} p_{12}^2 c_7 -p_{12} p_{13}^2 c_6
-p_{12}^3 c_9 \\ +\,p_{01}p_{13}^2 c_2 -p_{01}^2 p_{12} c_4
-p_{01}^2 p_{13} c_1-p_{13}^3 c_3+p_{01}^3 c_0.
\end{matrix}
\end{equation}
We write $I_{\rm ufv}$ for the ideal in $\QQ [p_{01}, p_{02}, \ldots, p_{23},\,c_0,c_1,\ldots,c_{19}]$
that is generated by these $20$ polynomials together with the Pl\"ucker quadric $\Pfaff(\cP)$.

The zero set of $I_{\rm ufv}$ in $\PP^5 \times \PP^{19}$  is 
the universal Fano variety of lines on cubic surfaces. We verified by computations on affine charts
that the~ideal $I_{\rm ufv}$ defines the correct scheme.
We consider  the \emph{tropical universal Fano variety}
\[ \Trop( I_{\rm ufv}) \quad \subset \quad \TP^5 \times \TP^{19}. \]

By the Structure Theorem  \cite[Theorem 3.3.5]{TropicalBook}, 
$\Trop( I_{\rm ufv}) $ is a pure $19$-di\-men\-si\-o\-nal balanced fan.
For simplicity, we disregard boundary phenomena, and we 
 replace each tropical projective space
$\TP^n$ with its dense tropical torus $\RR^{n+1}\!/\RR\vones$.
The former is compact while the latter is not. For a detailed discussion
see \cite[\S 6.2]{TropicalBook}. The points~in $\Trop( I_{\rm ufv})$
are the pairs consisting of a line in $\TP^3$ and a 
   cubic surface  that contains the line.
   The tropical line is represented by its Pl\"ucker vector $P   \in \RR^{6}$.
    The cubic is represented by its coefficient vector $C \in \RR^{20}$.
    Unlike in previous sections, this tropical cubic not be tropically smooth.
  A pair $(P,C)$ lies in $\Trop( I_{\rm ufv}) $ if and only if
  $\initial_{(P,C)}(I_{\rm ufv})$ contains no monomial.
  We take this initial ideal in the Laurent polynomial ring.

\begin{example}\label{ex:honeycomb:fano}
  The line given by $P = (26, 6, 17, 7, 18, 0)$ lies on the
   surface given~by $C=( 32, 17, 20, 41, 26, 17, 32, 33, 36, 54, 8, 1, 14, 4, 7, 18, 0, 0, 0, 0)$.
  This pair corresponds to the motif of type 3D in Example~\ref{ex:honeycomb:motif3D}; see the diagram on the left-hand side of Figure~\ref{fig:honeycomb:motif3D}.
  We verify the containment algebraically by checking that
\begin{equation}
\label{eq:inPC}
\begin{array}{ccl} \quad \initial_{(P,C)}(I_{\rm ufv}) & \,\,\,=\,\,\, & 
\langle \,p_{03} p_{12}-p_{02} p_{13}\,,\,
  p_{01} c_5-p_{12}c_8 \,, \,p_{13} c_{14}+p_{12} c_{15}\,, \,
   \,\,\,   \\ & &   p_{03} c_{14}+p_{02} c_{15}\,, \, p_{23} c_{15}+p_{13}c_{18}\,, \,
 p_{03} c_{11}+p_{13} c_{17}\,, \\ 
& &   \, p_{23} c_{14}-p_{12} c_{18}\,, \, p_{02} c_{11}+p_{12} c_{17}\,
      \rangle
\end{array}     
\end{equation}
contains no monomial.
 This initial ideal lives in the Laurent polynomial ring.
 % Of course, the initial forms
% of the $20$ new generators of $I_{\rm ufv}$ have bidegree $(3,1)$. 
For instance, the ten terms in (\ref{eq:mybihomo}) have weights
  $ 68, 68,73,75,75, 82, 85, 87, 95$, $110$ in this order, and the resulting
    initial form equals $ \,(p_{01} c_5-p_{12}c_8 )p_{12}p_{13}$.

The point $(P,C)$ lies in the relative interior of a maximal cell of $\Trop( I_{\rm ufv}) $.
The inequality description of this cell  is read off
from a Gr\"obner basis of $I_{\rm ufv}$. For instance, the polynomial
(\ref{eq:mybihomo}) contributes the equation
$\,P_{01}+ C_5 = P_{12} + C_8\,$ and eight inequalities, namely,
$P_{01}+C_5 +P_{12}+P_{13}$ is 
bounded above~by
\[
\begin{matrix}
 P_{01}+2  P_{12}+ C_{7}\,,\,\,
3  P_{12}+ C_{9}\,,\,\,
 P_{12}+2  P_{13}+ C_{6}\,,\,\,
 P_{01}+2  P_{13}+ C_{2}, \\
2  P_{01}+ P_{12}+ C_{4}\,,\,\,
2  P_{01}+ P_{13}+ C_{1}\,,\,\,
3  P_{13}+ C_{3}\,\,\, {\rm and} \,\,\,
3  P_{01}+ C_{0}.\end{matrix}
\]
Such constraints, derived from polynomials in $I_{\rm ufv}$,
define the cells of $\Trop( I_{\rm ufv}) $.
\end{example}

The maximal cones of  $\Trop( I_{\rm ufv}) $ represent occurrences 
of motifs in $3 \Delta_3$. In particular,
if we could compute this fan, then this would be an 
\emph{ab initio} derivation of the motifs 3A, 3B, $\ldots\,$, 3J.
These were found geometrically in~\cite{PV}.
% In the future, it 
% would be desirable to carry out such derivations by computer,
% via the algebraic framework in \cite{TropicalBook}.

\begin{remark} \label{rmk:identified}
Motifs and their occurrences can be identified from  initial ideals
$\initial_{(P,C)}(I_{\rm ufv})$. For instance, the indices
$i$ of the unknowns $c_i$ in  (\ref{eq:inPC})
form the list $(A,B,C,D,E,F,G,H) \! = (11,17,18,14,15,5,8)$ we saw 
in Example~\ref{ex:honeycomb:alternate3D}.
\end{remark}

Unfortunately, it is very difficult to compute with the ideal $I_{\rm ufv}$.
Even finding a single Gr\"obner basis is hard.
For instance, the computation of (\ref{eq:inPC})
 only terminated after we imposed some degree constraints in \macaulay.
One open problem naturally arising here is to find a tropical basis of $I_{\rm ufv}$.
% How far are the ideal generators, one of degree $(2,0) $ and twenty of degree $(3,1)$, from being a tropical basis? 
% The tropical prevariety they define strictly contains $\Trop( I_{\rm ufv}) $, but how different are these two?
The Schl\"afli fan fits into a
broader theory, yet to be developed, for discriminants
of morphisms in tropical algebraic geometry.
We propose the following approach.
Let $\mathcal{X}$ be a tropical variety in $\TP^d \times \TP^n$
and $\phi$ the projection from $\mathcal{X}$ onto the second factor $\TP^n$.
We assume that $\phi$ is onto, so  $\dim(\mathcal{X}) \geq n$.
Let $\mathcal{X}^{(n-1)}$ be the subcomplex of $\mathcal{X}$ consisting of all
cells of dimension at most $n-1$. 
 If this is the \emph{ramification locus}
 then $\phi(\mathcal{X}^{(n-1)})$  plays the role of the
 \emph{branch locus}. 
 
 \begin{example}
 \label{ex:tropdisc}
 Tropical discriminants  \cite{DFS} are a special
 case of this construction. Let $\mathcal{A}$ be a subset
 of $n+1$ elements in $\ZZ^d$. Consider hypersurfaces in $d$-space
 defined  by Laurent polynomials with these $n+1$ terms.
We write $C = (C_a:a\in\mathcal{A}) \in \RR^{n+1}$ for the vector of coefficients, 
and $P = (P_1,\ldots,P_d)$ for a point in  $\RR^d$.
The \emph{universal tropical hypersurface} is the tropical variety 
$\mathcal{X}$ defined by
\begin{equation}
\label{eq:troppoly}
 \bigoplus_{a \in \mathcal{A}} \, C_a \odot P^a \quad = \quad
\bigoplus_{a \in \mathcal{A}} \, C_a \odot P_1^{\odot a_1} \odot P_2^{\odot a_2} \odot
\cdots \odot P_d^{\odot a_d} . 
\end{equation}
The map $\phi :\mathcal{X} \rightarrow \RR^{n+1}/\RR {\bf 1}, (C,P) \mapsto C$ is surjective.
The fiber $\phi^{-1}(C)$ is the hypersurface in $\RR^d$
whose tropical polynomial has coefficients $C$.
The tropical variety $\mathcal{X}$ has dimension $n+d-1$. It is a fan with $\binom{n+1}{2}$ maximal cones, one for
each pair of terms in (\ref{eq:troppoly}). The subfan $\mathcal{X}^{(n-1)}$
consists of $\binom{n+1}{d+2}$ cones of dimension $n-1$. On each cone,
the minimum among the $n+1$ terms in (\ref{eq:troppoly}) is attained
by a fixed set of $d+2$ terms.
Hence the regular subdivision of $\mathcal{A}$ defined by $C$
is not a triangulation. The image $\phi(\mathcal{X}^{(n-1)})$  
consists of the~cones of codimension $\geq 1$ in the secondary fan of $\mathcal{A}$.
In particular, the tropical discriminant defined above contains that
of \cite{DFS}. The difference arises from~the distinction between the 
$\mathcal{A}$-discriminant and the principal $\mathcal{A}$-determinant; see~\cite{DT} and \cite{GKZ}.
% DIckenstein-Tabera (DCG)
\end{example}
 
\begin{remark}
 The number of cones in $\mathcal{X}^{(n-1)}$ is much smaller than that of its
 image under the projection $\phi$. This phenomenon is familiar from
 computer algebra (cf.~elimination theory) and optimization (cf.~extended formulations).
 In our context, take $\mathcal{A} = 3 \Delta_3$ in Example \ref{ex:tropdisc}.
 The universal cubic surface  $\mathcal{X}$ 
has only $\binom{20}{2} = 190$  maximal cones,
 whereas its discriminant   $\phi(\mathcal{X}^{(n-1)})$ forms the walls
 between many more than $ \numberOfTriangulations$ cones.
 \end{remark}
 
We now come to the main theoretical result in this section.
The role of points will be played by lines. An analogous
result holds for Fano varieties of arbitrary hypersurfaces (\ref{eq:troppoly}).
We focus on the case of cubic surfaces in $\TP^3$.

\begin{proposition}
  Let $\mathcal{X} = \Trop( I_{\rm ufv}) $ be the tropical universal Fano variety~in $\TP^5 {\times} \TP^{19}$ and $\phi$ the map onto the second factor (space of tropical cubics).
    The~tropical discriminant of $\phi$ is contained in the union of the codimension~$1$ cones in the Schl\"afli fan.
  The latter is a subset of the union of all Schl\"afli~walls.
\end{proposition}

\begin{proof}
All cubics $C$ in the interior of one fixed Schl\"afli cone
have the same visible motifs. The Pl\"ucker vectors $P$ of the $27$ lines
are linear functions in the entries of $C$, as long as $C$ staying within one Schl\"afli cone.
Hence the set of cells in $\mathcal{X}$ that are
intersected by the fiber $\phi^{-1}(C)$ remains constant throughout that
Schl\"afli cone. These cells all have the full dimension $19$. In particular,
$\phi^{-1}(C)$ is disjoint from $\mathcal{X}^{(18)}$ for $C$ 
in the interior of a Schl\"afli cone. This shows that this interior is
disjoint from the tropical discriminant of $\phi$.
\end{proof}

We conclude with a brief discussion of a
 related universal family. It lives in $\PP^3 \times \PP^{19}$,
where $\PP^3$ now parametrizes planes in the ambient $3$-space.
Each plane $\{u_0 x_0 + u_1 x_1 + u_2 x_2 + u_3 x_3 = 0\}$
intersects a cubic surface in a plane cubic curve. The plane is
a \emph{tritangent plane} if the plane cubic decomposes into three lines.
The \emph{universal Brill variety}  is the 
$19$-dimensional irreducible variety consisting of all pairs $(u,f)$,
where $u = (u_0:u_1:u_2:u_3)$ is a tritangent plane to the cubic surface $\{f = 0\}$.
The map from this variety onto $\PP^{19}$ is a $45$-to-$1$ covering,
since a general cubic surface has $45$ tritangent planes.

We introduce an ideal $I_{\rm bri}$  that defines the universal Brill variety.
It lives in the ring $\, \QQ [u_0,u_1,u_2,u_3,\,c_0,c_1,\ldots,c_{19}]$, 
where the last ten unknowns are the coefficients of a ternary cubic.
In these unknowns, we consider the prime ideal of
codimension $3$ and degree $15$ that defines the factorizable cubics.
Its variety is an instance of a \emph{Chow variety},
and the equations are known as \emph{Brill equations} \cite[\S I.4.H]{GKZ}.
This prime ideal is generated by $35$ quartics in the $10$ unknowns. 

We now derive $35$ generators of $I_{\rm bri}$.
Set $x_3 = -\frac{1}{u_3} (u_0 x_0 + u_1 x_1 + u_2 x_2)$ in 
 $f$, and  clear denominators to get a ternary cubic with
coefficients  $\QQ[u_0,u_1,u_2,u_3]$. We substitute these
cubics into the Brill equations and we remove factors of~$u_3$.
The resulting  $35$ polynomials of bidegree $(7,4)$ in $(u,c)$ generate our ideal $I_{\rm bri}$.

We are interested in the resulting \emph{tropical universal Brill variety}
\[ \mathcal{X} \,\, = \,\, \Trop( I_{\rm bri}) \quad \subset \quad \TP^3 \times \TP^{19}. \]
Its points are pairs consisting of a tropical cubic
and a tritangent plane. The maximal cones of 
$\Trop( I_{\rm bri}) $ represent occurrences of \emph{triple motifs}
in $3 \Delta_3$. 
It would be desirable to compute these.
We note that the tritangent planes correspond
to the $45$ triangles in the \emph{Schl\"afli graph}. This is the 
$10$-regular graph whose vertices
are the $27$ lines, and whose edges are incident pairs of lines.
The motifs and the triple motifs that occur in a triangulation 
can be seen as a tropical structure
 for annotating and extending the Schl\"afli graph.

\subsection*{Acknowledgements}
We are very grateful to Lars Kastner, Benjamin Lorenz and Andreas Paffenholz for their help with the computations for this project. 
We thank Sara Lamboglia, Yue Ren and Emre Sert\"oz for their comments on a manuscript version of this article. 
We are also grateful to the two anonymous referees whose comments helped us improving the article.  
Michael~Joswig was supported by Deutsche Forschungsgemeinschaft (EXC
2046: "MATH+", SFB-TRR 195: "Symbolic Tools in Mathematics and their
Application", and GRK 2434: "Facets of Complexity").

\bibliographystyle{spmpsci}
\bibliography{cubics}

\begin{thebibliography}{10}
\providecommand{\url}[1]{{#1}}
\providecommand{\urlprefix}{URL }
\expandafter\ifx\csname urlstyle\endcsname\relax
  \providecommand{\doi}[1]{DOI~\discretionary{}{}{}#1}\else
  \providecommand{\doi}{DOI~\discretionary{}{}{}\begingroup
  \urlstyle{rm}\Url}\fi

\bibitem{AltshulerSteinberg:1973}
Altshuler, A., Steinberg, L.: Neighborly {$4$}-polytopes with {$9$} vertices.
\newblock J. Combinatorial Theory Ser. A \textbf{15}, 270--287 (1973)

\bibitem{BogartKatz}
Bogart, T., Katz, E.: Obstructions to lifting tropical curves in surfaces in
  3-space.
\newblock SIAM J. Discrete Math. \textbf{26}(3), 1050--1067 (2012)

\bibitem{BrugalleShaw}
Brugall\'{e}, E., Shaw, K.: Obstructions to approximating tropical curves in
  surfaces via intersection theory.
\newblock Canad. J. Math. \textbf{67}(3), 527--572 (2015)

\bibitem{book:triangulations}
De~Loera, J.A., Rambau, J., Santos, F.: Triangulations, \emph{Algorithms and
  Computation in Mathematics}, vol.~25.
\newblock Springer-Verlag, Berlin (2010)

\bibitem{DFS}
Dickenstein, A., Feichtner, E.M., Sturmfels, B.: Tropical discriminants.
\newblock J. Amer. Math. Soc. \textbf{20}(4), 1111--1133 (2007)

\bibitem{DT}
Dickenstein, A., Tabera, L.F.: Singular tropical hypersurfaces.
\newblock Discrete Comput. Geom. \textbf{47}(2), 430--453 (2012)

\bibitem{flyspeck}
\url{https://github.com/flyspeck/flyspeck}

\bibitem{DMV:polymake}
Gawrilow, E., Joswig, M.: \polymake: a framework for analyzing convex
  polytopes.
\newblock In: Polytopes -- combinatorics and computation (Oberwolfach, 1997),
  \emph{DMV Sem.}, vol.~29, pp. 43--73. Birk\-h\"au\-ser (2000)

\bibitem{geiger}
Geiger, A.: On realizability of lines on tropical cubic surfaces and the
  {B}rundu-{L}ogar normal form.
\newblock Le Matematiche  (2020)

\bibitem{GKZ}
Gel'fand, I.M., Kapranov, M., Zelevinsky, A.: Discriminants, Resultants, and
  Multidimensional Determinants.
\newblock Birkh\"auser, Boston (1994)

\bibitem{scip}
Gleixner, A., et~al.: {The SCIP Optimization Suite 6.0}.
\newblock Technical report, Optimization Online (2018)

\bibitem{M2}
Grayson, D.R., Stillman, M.E.: \macaulay, a software system for research in
  algebraic geometry.
\newblock \urlprefix\url{http://www.math.uiuc.edu/Macaulay2/}

\bibitem{atint}
Hampe, S.: \atint: a \polymake extension for algorithmic tropical intersection
  theory.
\newblock European J. Combin. \textbf{36}, 579--607 (2014)

\bibitem{HampeJoswig:2017}
Hampe, S., Joswig, M.: Tropical computations in \polymake.
\newblock In: Algorithmic and experimental methods in algebra, geometry, and
  number theory, pp. 361--385. Springer, Cham (2017)

\bibitem{HampeJoswigSchroeter:MEGA2017}
Hampe, S., Joswig, M., Schr\"{o}ter, B.: Algorithms for tight spans and
  tropical linear spaces.
\newblock J. Symbolic Comput. \textbf{91}, 116--128 (2019)

\bibitem{mptopcom:paper}
Jordan, C., Joswig, M., Kastner, L.: Parallel enumeration of triangulations.
\newblock Electron. J. Combin. \textbf{25}(3), Paper 3.6, 27 (2018)

\bibitem{extension:tropcubics}
Joswig, M., Panizzut, M., Sturmfels, B.: \polymake extension \tropcubics.
\newblock \url{https://polymake.org/extensions/tropicalcubics}

\bibitem{TropicalBook}
Maclagan, D., Sturmfels, B.: Introduction to tropical geometry, \emph{Graduate
  Studies in Mathematics}, vol. 161.
\newblock American Mathematical Society, Providence, RI (2015)

\bibitem{nauty}
McKay, B., Piperno, A.: Practical graph isomorphism, {II}.
\newblock J. Symbolic Comput. \textbf{60}, 94--112 (2014)

\bibitem{polydb:paper}
Paffenholz, A.: \polydb: a database for polytopes and related objects.
\newblock In: Algorithmic and experimental methods in algebra, geometry, and
  number theory, pp. 533--547. Springer, Cham (2017)

\bibitem{sicily}
Panizzut, M., Sert\"oz, E., Sturmfels, B.: An octanomial model for cubic
  surfaces.
\newblock Le Matematiche  (2020)

\bibitem{PV}
Panizzut, M., Vigeland, M.: Tropical lines on smooth tropical surfaces (2019).
\newblock \arXiv{0708.3847}

\bibitem{Schl}
Schl\"afli, L.: An attempt to determine the twenty-seven lines upon a surface
  of the third order, and to divide such surfaces into species in reference to
  the reality of the lines upon the surface.
\newblock Quarterly Journal of Pure and Applied Mathematics \textbf{2}, 55--65
  (1858)

\bibitem{VigelandInfinite}
Vigeland, M.: Smooth tropical surfaces with infinitely many tropical lines.
\newblock Ark. Mat. \textbf{48}(1), 177--206 (2010)

\end{thebibliography}
  
\end{document}